\documentclass[11pt]{article}
\usepackage{epsf}
\usepackage{amsmath}
\usepackage{epsfig}
\usepackage{times}
\usepackage{amssymb}
\usepackage{amsthm}
\usepackage{setspace}
\usepackage{cite}

\usepackage{algorithmic}  
\usepackage{algorithm}

\usepackage{shadow}
\usepackage{fancybox}
\usepackage{fancyhdr}

\def\x{{\bf x}}

\def\x{{\mathbf x}}

\def\x{{\bf x}}

\def\d{{\bf d}}

\def\h{{\bf h}}

\def\diag{{\rm diag}\,}

\def\be{\begin{equation}}
\def\ee{\end{equation}}
\def\ba{\left[\begin{array}}
\def\ea{\end{array}\right]}

\def\x{{\bf x}}

\def\d{{\bf d}}

\def\1{{\bf 1}}

\def\diag{{\mbox{diag}}}

\def\g{{\bf g}}
\def\0{{\bf 0}}

\newtheorem{theorem}{Theorem}

\newtheorem{lemma}{Lemma}

\begin{document}

\begin{singlespace}

\title {Spherical perceptron as a storage memory with limited errors 
}
\author{
\textsc{Mihailo Stojnic}
\\
\\
{School of Industrial Engineering}\\
{Purdue University, West Lafayette, IN 47907} \\
{e-mail: {\tt mstojnic@purdue.edu}} }
\date{}
\maketitle

\centerline{{\bf Abstract}} \vspace*{0.1in}

It has been known for a long time that the classical spherical perceptrons can be used as storage memories. Seminal work of Gardner, \cite{Gar88}, started an analytical study of perceptrons storage abilities. Many of the Gardner's predictions obtained through statistical mechanics tools have been rigorously justified. Among the most important ones are of course the storage capacities. The first rigorous confirmations were obtained in \cite{SchTir02,SchTir03} for the storage capacity of the so-called positive spherical perceptron. These were later reestablished in \cite{TalBook} and a bit more recently in \cite{StojnicGardGen13}. In this paper we consider a variant of the spherical perceptron that operates as a storage memory but allows for a certain fraction of errors. In Gardner's original work the statistical mechanics predictions in this directions were presented sa well. Here, through a mathematically rigorous analysis, we confirm that the Gardner's predictions in this direction are in fact provable upper bounds on the true values of the storage capacity. Moreover, we then present a mechanism that can be used to lower these bounds. Numerical results that we present indicate that the Garnder's storage capacity predictions may, in a fairly wide range of parameters, be not that far away from the true values.

\vspace*{0.25in} \noindent {\bf Index Terms: Perceptron with errors; storage capacity}.

\end{singlespace}

\section{Introduction}
\label{sec:back}

In this paper we will study a special type of the classical spherical perceptron problem. Of course, spherical perceptrons are a well studied class of problems with applications in various fields, ranging from neural networks and statistical physics/mechanics to high-dimensional geometry and biology. While the spherical perceptron like problems had been known for a long time (for various mathematical versions see, e.g. \cite{Schlafli,Cover65,Winder,Winder61,Wendel62,Cameron60,Joseph60,BalVen87,Ven86}), it is probably the work of Gardner \cite{Gar88} that brought them in the research spotlight. One would be inclined to believe that the main reason for that was Gardner's ability to quantify many of the features of the spherical perceptrons that were not so easy to handle through the standard mathematical tools typically used in earlier works. Namely, in \cite{Gar88}, Gardner introduced a fairly neat type of analysis based on a statistical mechanics approach typically called the replica theory. As a result she was able to quantify almost any of the spherical perceptrons typical features of interest. While some of the results she obtained were known (for example, the storage capacity with zero-thresholds, see, e.g. \cite{Schlafli,Cover65,Winder,Winder61,Wendel62,Cameron60,Joseph60,BalVen87,Ven86}) many other ones were not (storage capacity with non-zero thresholds, typical volume of interactions strengths for which the memory functions properly, the storage capacities of memories with errors, and so on). Moreover, many of the results that she obtained remained as mathematical conjectures (either in the form of those related to quantities which are believed to be the exact predictions or in the form of those related to quantities which may be solid approximations). In recent years some of those that had been believed to be exact have indeed been rigorously proved (see, e.g. \cite{SchTir02,SchTir03,TalBook,StojnicGardGen13}) whereas many of those that are believed to be solid approximations have been shown to be at the very least rigorous bounds (see, e.g. \cite{StojnicGardGen13,StojnicGardSphNeg13}).

In this paper we will also look at one of the features of the spherical perceptron. The quantity that we will be interested in this paper in particular is fairly closely related to the well-known storage capacity. Namely, we will indeed attempt to evaluate the storage capacity of the spherical perceptron, however, instead of insisting that all the patterns should be memorized correctly we will also allow for a certain fraction of errors. In other words, we we will allow that a certain fraction of patterns can in fact be memorized incorrectly. Throughout the paper, we will often refer to the capacity of such a memory as the storage capacity with errors. Of course, this problem was already studied in \cite{Gar88} and a nice set of observations related to it has already been made there. Here, we will through a mathematically rigorous analysis attempt to confirm many of them.

Before going into the details of our approach we will recall on the basic definitions related to the spherical perceptron and needed for its analysis. Also, to make the presentation easier to follow we find it useful to briefly sketch how the rest of the paper is organized. In Section \ref{sec:mathsetupper} we will, as mentioned above, introduce a more formal mathematical description of how a perceptron operates. In Section \ref{sec:knownres} we will present several results that are known for the classical spherical perceptron. In Section \ref{sec:sphpererrors} we will discuss the storage capacity when the errors are allowed. We will recall on the known results and later on in Section \ref{sec:sphpererrorsrig} present a powerful mechanism that can be used to prove that many of the known results are actually rigorous bounds on the quantities of interest. In Section \ref{sec:sphpererrorslow} we will then present a further refinement of the mechanism from Section \ref{sec:sphpererrorsrig} that can be used to potentially lower the values of the storage capacity obtained in Section \ref{sec:sphpererrorsrig}. Finally, in Section \ref{sec:conc} we will discuss obtained results and present several concluding remarks.

\section{Mathematical setup of a perceptron}
\label{sec:mathsetupper}

To make this part of the presentation easier to follow we will try to introduce all important features of the spherical perceptron that we will need here by closely following what was done in \cite{Gar88} (and for that matter in our recent work \cite{StojnicGardGen13,StojnicGardSphNeg13}). So, as in \cite{Gar88}, we start with the following dynamics:
\begin{equation}
H_{ik}^{(t+1)}=\mbox{sign} (\sum_{j=1,j\neq k}^{n}H_{ij}^{(t)}X_{jk}-T_{ik}).\label{eq:defdyn}
\end{equation}
Following \cite{Gar88} for any fixed $1\leq i\leq m$ we will call each $H_{ij},1 \leq j\leq n $, the icing spin, i.e. $H_{ij}\in\{-1,1\},\forall i,j$. Continuing further with following \cite{Gar88}, we will call $X_{jk},1\leq j\leq n$, the interaction strength for the bond from site $j$ to site $i$. To be in a complete agreement with \cite{Gar88}, we in (\ref{eq:defdyn}) also introduced quantities $T_{ik},1\leq i\leq m,1\leq k\leq n$. $T_{ik}$is typically called the threshold for site $k$ in pattern $i$. However, to make the presentation easier to follow, we will typically assume that $T_{ik}=0$. Without going into further details we will mention though that all the results that we will present below can be easily modified so that they include scenarios where $T_{ik}\neq 0$.

Now, the dynamics presented in (\ref{eq:defdyn}) works by moving from a $t$ to $t+1$ and so on (of course one assumes an initial configuration for say $t=0$). Moreover, the above dynamics will have a fixed point if say there are strengths $X_{jk},1\leq j\leq n,1\leq k\leq m$, such that for any $1\leq i\leq m$
\begin{eqnarray}
& & H_{ik}\mbox{sign} (\sum_{j=1,j\neq k}^{n}H_{ij}X_{jk}-T_{ik})=1\nonumber \\
& \Leftrightarrow & H_{ik}(\sum_{j=1,j\neq k}^{n}H_{ij}X_{jk}-T_{ik})>0,1\leq j\leq n,1\leq k\leq n.\label{eq:defdynfp}
\end{eqnarray}
Of course, the above is a well known property of a very general class of dynamics. In other words, unless one specifies the interaction strengths the generality of the problem essentially makes it easy. After considering the general scenario introduced above, \cite{Gar88} then proceeded and specialized it to a particular case which amounts to including spherical restrictions on $X$. A more mathematical description of such restrictions considered in \cite{Gar88} essentially boils down to the following constraints
\begin{equation}
\sum_{j=1}^{n}X_{ji}^2=1,1\leq i\leq n.\label{eq:cosntX}
\end{equation}
The fundamental question that one typically considers then is the so-called storage capacity of the above dynamics or alternatively a neural network that it would represent (of course this is exactly one of the questions considered in \cite{Gar88}). Namely, one then asks how many patterns $m$ ($i$-th pattern being $H_{ij},1\leq j\leq n$) one can store so that there is an assurance that they are stored in a stable way. Moreover, since having patterns being fixed points of the above introduced dynamics is not enough to insure having a finite basin of attraction one often may impose a bit stronger threshold condition
\begin{eqnarray}
& & H_{ik}\mbox{sign} (\sum_{j=1,j\neq k}^{n}H_{ij}X_{jk}-T_{ik})=1\nonumber \\
& \Leftrightarrow & H_{ik}(\sum_{j=1,j\neq k}^{n}H_{ij}X_{jk}-T_{ik})>\kappa,1\leq j\leq n,1\leq k\leq n,\label{eq:defdynfpstr}
\end{eqnarray}
where typically $\kappa$ is a positive number. We will refer to a perceptron governed by the above dynamics and coupled with the spherical restrictions and a positive threshold $\kappa$ as the positive spherical perceptron. Alternatively, when $\kappa$ is negative we will refer to it as the negative spherical perceptron (such a perceptron may be more of an interest from a purely mathematical point of view rather than as a neural network concept; nevertheless we will view it as an interesting mathematical problem; consequently, we will on occasion, in addition to the results that we will present for the standard positive perceptron, present quite a few results related to the negative case as well).

Also, we should mentioned that beyond the above mentioned negative case many other variants of the model that we study here are possible from a purely mathematical perspective. Moreover, many of them have found applications in various other fields as well. For example, a nice set of references that contains a collection of results related to various aspects of different neural networks models and their bio- and many other applications is  \cite{AgiAnnBarCooTan13a,AgiAnnBarCooTan13b,AgiBarBarGalGueMoa12,AgiBarGalGueMoa12,AgiAstBarBurUgu12,StojnicAsymmLittBnds11,BruParRit92}.

\section{Standard spherical perceptron --  known results}
\label{sec:knownres}

As mentioned above, our main interest in this paper will be a particular type of the spherical perceptron, namely the one that functions as a memory with a limited fraction of errors. However, before proceeding with the problem that we will study here in great detail we find it useful to first recall on several results known for the standard spherical perceptron, i.e. the one that functions as a storage memory without errors. That way it will be easier to properly position the results we intend to present here within the scope of what is already known.

\subsection{Statistical mechanics}
\label{sec:statmech}

We of course start with recalling on what was presented in \cite{Gar88}. In \cite{Gar88} a replica type of approach was designed and based on it a characterization of the storage capacity was presented. Before showing what exactly such a characterization looks like we will first formally define it. Namely, throughout the paper we will assume the so-called linear regime, i.e. we will consider the so-called \emph{linear} scenario where the length and the number of different patterns, $n$ and $m$, respectively are large but proportional to each other. Moreover, we will denote the proportionality ratio by $\alpha$ (where $\alpha$ obviously is a constant independent of $n$) and will set
\begin{equation}
m=\alpha n.\label{eq:defmnalpha}
\end{equation}
Now, assuming that $H_{ij},1\leq i\leq m,1\leq j\leq n$, are i.i.d. symmetric Bernoulli random variables, \cite{Gar88}, using the replica approach, gave the following estimate for $\alpha$ so that (\ref{eq:defdynfpstr}) holds with overwhelming probability (under overwhelming probability we will in this paper assume a probability that is no more than a number exponentially decaying in $n$ away from $1$)
\begin{equation}
\alpha_c(\kappa)=(\frac{1}{\sqrt{2\pi}}\int_{-\kappa}^{\infty}(z+\kappa)^2e^{-\frac{z^2}{2}}dz)^{-1}.\label{eq:garstorcap}
\end{equation}
Based on the above characterization one then has that $\alpha_c$ achieves its maximum over positive $\kappa$'s as $\kappa\rightarrow 0$. One in fact easily then has
\begin{equation}
\lim_{\kappa\rightarrow 0}\alpha_c(\kappa)=2.\label{eq:garstorcapk0}
\end{equation}
Also, to be completely exact, in \cite{Gar88}, it was predicted that the storage capacity relation from (\ref{eq:garstorcap}) holds for the range $\kappa\geq 0$.

\subsection{Rigorous results -- positive spherical perceptron ($\kappa\geq 0$)}
\label{sec:posphper}

The result given in (\ref{eq:garstorcapk0}) is of course well known and has been rigorously established either as a pure mathematical fact or even in the context of neural networks and pattern recognition \cite{Schlafli,Cover65,Winder,Winder61,Wendel62,Cameron60,Joseph60,BalVen87,Ven86}. In a more recent work \cite{SchTir02,SchTir03,TalBook} the authors also considered the storage capacity of the spherical perceptron and established that when $\kappa\geq 0$ (\ref{eq:garstorcap}) also holds. In our own work \cite{StojnicGardGen13} we revisited the storage capacity problems and presented an alternative mathematical approach that was also powerful enough to reestablish the storage capacity prediction given in (\ref{eq:garstorcap}). We below formalize the results obtained in \cite{SchTir02,SchTir03,TalBook,StojnicGardGen13}.

\begin{theorem} \cite{SchTir02,SchTir03,TalBook,StojnicGardGen13}
Let $H$ be an $m\times n$ matrix with $\{-1,1\}$ i.i.d.Bernoulli components. Let $n$ be large and let $m=\alpha n$, where $\alpha>0$ is a constant independent of $n$. Let $\alpha_c$ be as in (\ref{eq:garstorcap}) and let $\kappa\geq 0$ be a scalar constant independent of $n$. If $\alpha>\alpha_c$ then with overwhelming probability there will be no $\x$ such that $\|\x\|_2=1$ and (\ref{eq:defdynfpstr}) is feasible. On the other hand, if $\alpha<\alpha_c$ then with overwhelming probability there will be an $\x$ such that $\|\x\|_2=1$ and (\ref{eq:defdynfpstr}) is feasible.
\label{thm:SchTirTalSto}
\end{theorem}
\begin{proof}
Presented in various forms in \cite{SchTir02,SchTir03,TalBook,StojnicGardGen13}.
\end{proof}

As mentioned earlier, the results given in the above theorem essentially settle the storage capacity of the positive spherical perceptron or the Gardner problem. However, there are a couple of facts that should be pointed out (emphasized):

1) The results presented above relate to the \emph{positive} spherical perceptron. It is not clear at all if they would automatically translate to the case of the negative spherical perceptron. As we hinted earlier, the case of the negative spherical perceptron ($\kappa<0$) may be more of interest from a purely mathematical point of view than it is from say the neural networks point of view. Nevertheless, such a mathematical problem may turn out to be a bit harder than the one corresponding to the standard positive case. In fact, in \cite{TalBook}, Talagrand conjectured (conjecture 8.4.4) that the above mentioned $\alpha_c$ remains an upper bound on the storage capacity even when $\kappa<0$, i.e. even in the case of the negative spherical perceptron. However, he does seem to leave it as an open problem what the exact value of the storage capacity in the negative case should be. In our own work \cite{StojnicGardGen13} we confirmed this Talagrand's conjecture and showed that even in the negative case $\alpha_c$ from (\ref{eq:garstorcap}) is indeed an upper bound on the storage capacity.

2) It is rather clear but we do mention that the overwhelming probability statement in the above theorem is taken with respect to the randomness of $H$. To analyze the feasibility of (\ref{eq:defprobucor1}) we in \cite{StojnicGardGen13} relied on a mechanism we recently developed for studying various optimization problems in \cite{StojnicRegRndDlt10}. Such a mechanism works for various types of randomness. However, the easiest way to present it was assuming that the underlying randomness is standard normal. So to fit the feasibility of (\ref{eq:defprobucor1}) into the framework of \cite{StojnicRegRndDlt10} we in \cite{StojnicGardGen13} formally assumed that the elements of matrix $H$ are i.i.d. standard normals. In that regard then what was proved in \cite{StojnicGardGen13} is a bit different from what was stated in the above theorem. However, as mentioned in \cite{StojnicGardGen13} (and in more detail in \cite{StojnicRegRndDlt10,StojnicMoreSophHopBnds10}) all our results from \cite{StojnicGardGen13} continue to hold for a very large set of types of randomness and certainly for the Bernouilli one assumed in Theorem \ref{thm:SchTirTalSto}.

3) We will continue to call the critical value of $\alpha$ so that (\ref{eq:defdynfpstr}) is feasible the storage capacity even when $\kappa<0$, even though it may be linguistically a bit incorrect, given the neural network interpretation of finite basins of attraction mentioned above.

\subsection{Rigorous results -- negative spherical perceptron ($\kappa< 0$)}
\label{sec:negphper}

In our recent work \cite{StojnicGardSphNeg13} we went a step further and considered the negative version of the standard spherical perceptron. While the results that we will present later on in Sections \ref{sec:sphpererrorsrig} and \ref{sec:sphpererrorslow} will be valid for any $\kappa$ our main concern will be from a neural network point of view and as such will be related to the positive case, i.e. to $\kappa\geq 0$ scenario. In that regard the results that we review in this subsection may seem as not as important as those from the previous subsections. However, once we present the main results in Sections \ref{sec:sphpererrorsrig} and \ref{sec:sphpererrorslow} it will be clear that there is an interesting conceptual similarity that is deeply rooted in a combinatorial similarity of what we will present in this subsection (and what was essentially proved in \cite{StojnicGardGen13,StojnicGardSphNeg13}) and the results that we will present in Sections \ref{sec:sphpererrorsrig} and \ref{sec:sphpererrorslow}.

As mentioned above under point 3), we in \cite{StojnicGardSphNeg13} called the corresponding limiting $\alpha$ in $\kappa<0$ case the storage capacity of the negative spherical perceptron. Before presenting the storage capacity results that we obtained in \cite{StojnicGardGen13,StojnicGardSphNeg13} we will find it useful to slightly redefine the original feasibility problem considered above. This will of course be of a great use in the exposition that will follow as well.

We first recall that in \cite{StojnicGardSphNeg13} we studied the so-called uncorrelated case of the spherical perceptron (more on an equally important correlated case can be found in e.g. \cite{StojnicGardGen13,Gar88}). This is the same scenario that we will study here (so the simplifications that we made in \cite{StojnicGardSphNeg13} and that we are about to present below will be in place later on as well). In the uncorrelated case, one views all patterns $H_{i,1:n},1\leq i\leq m$, as uncorrelated (as expected, $H_{i,1:n}$ stands for vector $[H_{i1},H_{i2},\dots,H_{in}]$). Now, the following becomes the corresponding version of the question of interest mentioned above: assuming that $H$ is an $m\times n$ matrix with i.i.d. $\{-1,1\}$ Bernoulli entries and that $\|\x\|_2=1$, how large $\alpha=\frac{m}{n}$ can be so that the following system of linear inequalities is satisfied with overwhelming probability
\begin{equation}
H\x\geq \kappa.\label{eq:defprobucor}
\end{equation}
This of course is the same as if one asks how large $\alpha$ can be so that the following optimization problem is feasible with overwhelming probability
\begin{eqnarray}
& & H\x\geq \kappa\nonumber \\
& & \|\x\|_2=1.\label{eq:defprobucor1}
\end{eqnarray}
To see that (\ref{eq:defprobucor}) and (\ref{eq:defprobucor1}) indeed match the above described fixed point condition it is enough to observe that due to statistical symmetry one can assume $H_{i1}=1,1\leq i\leq m$. Also the constraints essentially decouple over the columns of $X$ (so one can then think of $\x$ in (\ref{eq:defprobucor}) and (\ref{eq:defprobucor1}) as one of the columns of $X$). Moreover, the dimension of $H$ in (\ref{eq:defprobucor}) and (\ref{eq:defprobucor1}) should be changed to $m\times (n-1)$; however, since we will consider a large $n$ scenario to make writing easier we keep the dimension as $m\times n$. Also, as mentioned under point 2) above, we will, without a loss of generality, treat $H$ in (\ref{eq:defprobucor1}) as if it has i.i.d. standard normal components. Moreover, in \cite{StojnicGardGen13} we also recognized that (\ref{eq:defprobucor1}) can be rewritten as the following optimization problem
\begin{eqnarray}
\xi_n=\min_{\x} \max_{\lambda\geq 0} & &  \kappa\lambda^T\1- \lambda^T H\x \nonumber \\
\mbox{subject to} & & \|\lambda\|_2= 1\nonumber \\
& & \|\x\|_2=1,\label{eq:uncorminmax}
\end{eqnarray}
where $\1$ is an $m$-dimensional column vector of all $1$'s. Clearly, if $\xi_n\leq 0$ then (\ref{eq:defprobucor1}) is feasible. On the other hand, if $\xi_n>0$ then (\ref{eq:defprobucor1}) is not feasible. That basically means that if we can probabilistically characterize the sign of $\xi_n$ then we could have a way of determining $\alpha$ such that $\xi_n\leq 0$. That is exactly what we have done in \cite{StojnicGardGen13} on an ultimate level for $\kappa\geq 0$ and on a say upper-bounding level for $\kappa<0$. Of course, we do mention again, that as far as point 2) goes, we in \cite{StojnicGardSphNeg13} (and will in this paper as well) without loss of generality again made the same type of assumption that we had made in \cite{StojnicGardGen13} related to the statistics of $H$. In other words, as far as the presentation below is concerned, we will continue to assume that the elements of matrix $H$ are i.i.d. standard normals (as mentioned above, such an assumption changes nothing in the validity of the results that we will present; also, more on this topic can be found in e.g. \cite{StojnicHopBnds10,StojnicLiftStrSec13,StojnicRegRndDlt10} where we discussed it a bit further). Relying on the strategy developed in \cite{StojnicRegRndDlt10,StojnicGorEx10} and on a set of results from \cite{Gordon85,Gordon88} we in \cite{StojnicGardGen13} proved the following theorem that essentially extends Theorem \ref{thm:SchTirTalSto} to the $\kappa<0$ case and thereby resolves Conjecture 8.4.4 from \cite{TalBook} in positive:

\begin{theorem} \cite{StojnicGardGen13}
Let $H$ be an $m\times n$ matrix with i.i.d. standard normal components. Let $n$ be large and let $m=\alpha n$, where $\alpha>0$ is a constant independent of $n$. Let $\xi_n$ be as in (\ref{eq:uncorminmax}) and let $\kappa$ be a scalar constant independent of $n$. Let all $\epsilon$'s be arbitrarily small constants independent of $n$. Further, let $\g_i$ be a standard normal random variable and set
\begin{equation}
f_{gar}(\kappa)=\frac{1}{\sqrt{2\pi}}\int_{-\kappa}^{\infty}(\g_i+\kappa)^2e^{-\frac{\g_i^2}{2}}d\g_i.\label{eq:fgarlemmaunncorlb}
\end{equation}
Let $\xi_n^{(l)}$ and $\xi_n^{(u)}$ be scalars such that
\begin{eqnarray}
(1-\epsilon_{1}^{(m)})\sqrt{\alpha f_{gar}(\kappa)}-(1+\epsilon_{1}^{(n)})-\epsilon_{5}^{(g)} & > & \frac{\xi_n^{(l)}}{\sqrt{n}}\nonumber \\
(1+\epsilon_{1}^{(m)})\sqrt{\alpha f_{gar}(\kappa)}-(1-\epsilon_{1}^{(n)})+\epsilon_{5}^{(g)} & < & \frac{\xi_n^{(u)}}{\sqrt{n}}.\label{eq:condxinthmstoc30}
\end{eqnarray}
If $\kappa\geq 0$ then
\begin{equation}
 \lim_{n\rightarrow\infty}P(\xi_n^{(l)}\leq \xi_n\leq \xi_n^{(u)})=\lim_{n\rightarrow\infty}P(\min_{\|\x\|_2=1}\max_{\|\lambda\|_2=1,\lambda_i\geq 0}(\xi_n^{(l)}\leq \kappa\lambda^T\1-\lambda^TH\x)\leq \xi_n^{(u)})\geq 1. \label{eq:probthmstoc30poskappa}
\end{equation}
Moreover, if $\kappa< 0$ then
\begin{equation}
 \lim_{n\rightarrow\infty}P(\xi_n\geq \xi_n^{(l)})=\lim_{n\rightarrow\infty}P(\min_{\|\x\|_2=1}\max_{\|\lambda\|_2=1,\lambda_i\geq 0}( \kappa\lambda^T\1-\lambda^TH\x)\geq \xi_n^{(u)})\geq 1. \label{eq:probthmstoc30negkappa}
\end{equation}
\label{thm:Stoc30}
\end{theorem}
\begin{proof}
Presented in \cite{StojnicGardGen13}.
\end{proof}
In a more informal language (essentially ignoring all technicalities and $\epsilon$'s) one has that as long as
\begin{equation}
\alpha>\frac{1}{f_{gar}(\kappa)},\label{eq:condalphauncorlb}
\end{equation}
the problem in (\ref{eq:defprobucor1}) will be infeasible with overwhelming probability. On the other hand, one has that when $\kappa\geq 0$ as long as
\begin{equation}
\alpha<\frac{1}{f_{gar}(\kappa)},\label{eq:condalphauncorubpos}
\end{equation}
the problem in (\ref{eq:defprobucor1}) will be feasible with overwhelming probability. This of course settles the case $\kappa \geq 0$ completely and essentially establishes the storage capacity as $\alpha_c$ which of course matches the prediction given in the introductory analysis presented in \cite{Gar88} and of course rigorously confirmed by the results of \cite{SchTir02,SchTir03,TalBook}. On the other hand, when $\kappa < 0$ it only shows that the storage capacity with overwhelming probability is not higher than the quantity given in \cite{Gar88}. As mentioned above this confirms Talagrand's conjecture 8.4.4 from \cite{TalBook}. However, it does not settle problem (question) 8.4.2 from \cite{TalBook}.

The results obtained based on the above theorem as well as those obtained based on Theorem \ref{thm:SchTirTalSto} are presented in Figure \ref{fig:alfackappa}. When $\kappa\geq 0$ (i.e. when $\alpha\leq 2$) the curve indicates the exact breaking point between the ``overwhelming" feasibility and infeasibility of (\ref{eq:defprobucor1}). On the other hand, when $\kappa< 0$ (i.e. when $\alpha> 2$) the curve is only an upper bound on the storage capacity, i.e. for any value of the pair $(\alpha,\kappa)$ that is above the curve given in Figure \ref{fig:alfackappa}, (\ref{eq:defprobucor1}) is infeasible with overwhelming probability.
\begin{figure}[htb]
\centering
\centerline{\epsfig{figure=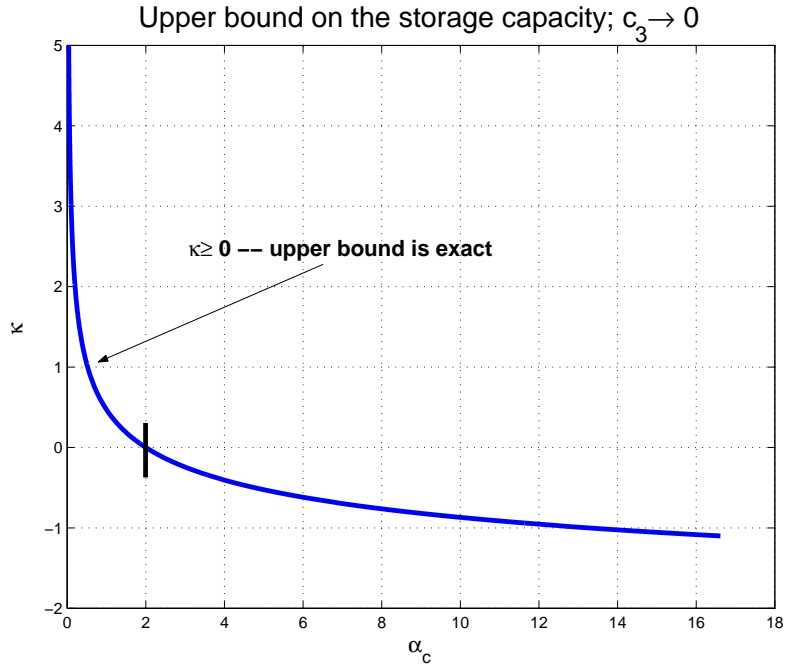,width=10.5cm,height=9cm}}
\caption{$\alpha_c$ as a function of $\kappa$}
\label{fig:alfackappa}
\end{figure}

Since the case $\kappa<0$ did not appear as settled based on the above presented results we then in \cite{StojnicGardSphNeg13} attempted to lower the upper bounds given in Theorem \ref{thm:Stoc30}. We created a fairly powerful mechanism that produced the following theorem as a way of characterizing the storage capacity of the negative spherical perceptron.

\begin{theorem}
Let $H$ be an $m\times n$ matrix with i.i.d. standard normal components. Let $n$ be large and let $m=\alpha n$, where $\alpha>0$ is a constant independent of $n$. Let $\kappa<0$ be a scalar constant independent of $n$. Let all $\epsilon$'s be arbitrarily small constants independent of $n$. Set
\begin{equation}
\widehat{\gamma^{(s)}}=\frac{2c_3^{(s)}+\sqrt{4(c_3^{(s)})^2+16}}{8},\label{eq:gamaliftthm}
\end{equation}
and
\begin{equation}
I_{sph}(c_3^{(s)}) = \widehat{\gamma^{(s)}}-\frac{1}{2c_3^{(s)}}\log(1-\frac{c_3^{(s)}}{2\widehat{\gamma^{(s)}}}).\label{eq:Isphthm}
\end{equation}
Set
\begin{equation}
p  =  1+\frac{c_3^{(s)}}{2\gamma_{per}^{(s)}},
q  =  \frac{c_3^{(s)}\kappa}{2\gamma_{per}^{(s)}},
r  =  \frac{c_3^{(s)}\kappa^2}{4\gamma_{per}^{(s)}},
s  =  -\kappa\sqrt{p}+\frac{q}{\sqrt{p}},
C  =  \frac{exp(\frac{q^2}{2p}-r)}{\sqrt{p}},\label{eq:helpdef}
\end{equation}
and
\begin{equation}
I_{per}^{(1)}(c_3^{(s)},\gamma_{per}^{(s)},\kappa)=\frac{1}{2}erfc(\frac{\kappa}{\sqrt{2}})+\frac{C}{2}(erfc(\frac{s}{\sqrt{2}})).\label{eq:Iper1}
\end{equation}
Further, set
\begin{equation}
I_{per}(c_3^{(s)},\alpha,\kappa)=\max_{\gamma_{per}^{(s)}\geq 0}(\gamma_{per}^{(s)}+\frac{1}{c_3^{(s)}}\log(I_{per}^{(1)}(c_3^{(s)},\gamma_{per}^{(s)},\kappa))).\label{eq:Iperthm}
\end{equation}
If $\alpha$ is such that
\begin{equation}
\min_{c_3^{(s)}\geq 0}(-\frac{c_3^{(s)}}{2}+I_{sph}(c_3^{(s)})+I_{per}(c_3^{(s)},\alpha,\kappa))<0,\label{eq:condliftsphnegthm}
\end{equation}
then (\ref{eq:defprobucor1}) is infeasible with overwhelming probability.
\label{thm:liftnegsphper}
\end{theorem}
\begin{proof}
Presented in \cite{StojnicGardSphNeg13}.
\end{proof}

The results one can obtain for the storage capacity based on the above theorem are presented in Figure \ref{fig:liftsphneg} (as mentioned in \cite{StojnicGardSphNeg13}, due to numerical optimizations involved the results presented in Figure \ref{fig:liftsphneg} should be taken only as an illustration; also as discussed in \cite{StojnicGardSphNeg13} taking $c_3^{(s)}\rightarrow 0$ in Theorem \ref{thm:liftnegsphper} produces the results of Theorem \ref{thm:Stoc30}). Even as such, they indicate that a visible improvement in the values of the storage capacity may be possible, though in a range of values of $\alpha$ substantially larger than $2$ (i.e. in a range of $\kappa$'s somewhat smaller than zero). While at this point this observation may look as unrelated to the problem that we will consider in the following section one should keep it in mind (essentially, a conceptually similar conclusion will be made later on when we study the capacities with limited errors).

\begin{figure}[htb]
\centering
\centerline{\epsfig{figure=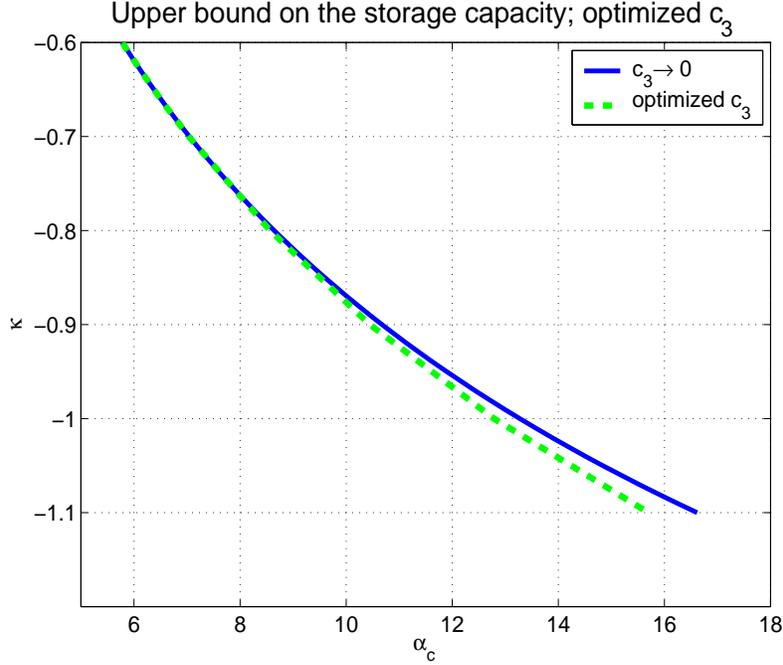,width=10.5cm,height=9cm}}
\caption{$\kappa$ as a function of $\alpha_{c}^{(u,low)}$}
\label{fig:liftsphneg}
\end{figure}

\section{Spherical perceptron with errors}
\label{sec:sphpererrors}

What we described in the previous section is a typical setup of a standard spherical perceptron. To be a bit more precise, it is a setup one can use to in a way quantify the storage capacity of the standard spherical perceptron. In this section we will slightly change this standard notion of how the spherical perceptron operates. In fact, what we will change will actually be what is an acceptable way of spherical perceptron's operation. Of course, such a chnage is not our invention. While it had been known for a long time, it is the work of Gardner \cite{Gar88} that popularized its an analytical study. Before, we present the known analytical predictions we will briefly sketch the main idea behind the spherical perceptrons that will be allowed to function as memories with errors. We will rely on many simplifications of the original perceptron setup from Section \ref{sec:mathsetupper} introduced in \cite{StojnicGardGen13,StojnicGardSphNeg13} and presented in Section \ref{sec:knownres}.

To that end we start by recalling that for all practical purposes needed here (and those we needed in \cite{StojnicGardGen13,StojnicGardSphNeg13}) the storage capacity of the standard spherical perceptron can be considered through the feasibility problem given in (\ref{eq:defprobucor1}) which we restate below
\begin{eqnarray}
& & H\x\geq \kappa\nonumber \\
& & \|\x\|_2=1.\label{eq:defprobucor1err}
\end{eqnarray}
We of course recall as well, that as argued in \cite{StojnicGardGen13,StojnicGardSphNeg13} (and as mentioned in the previous section) one can assume that the elements of $H$ are i.i.d. standard normals and that the dimension of $H$ is $m\times n$, where as earlier we keep the linear regime, i.e. continue to assume that $m=\alpha n$ where $\alpha$ is a constant independent of $n$. Now, if all inequalities in (\ref{eq:defprobucor1err}) are satisfied one can have that the dynamics established will be stable and all $m$ patterns could be successfully stored. On the other hand if one relaxes such a constraint so that only a fraction of them (say larger than $(1-f_{wb})$) is satisfied then only such a fraction of patterns could be successfully stored (of course one views storage at each site $i$; however, due to symmetry as discussed earlier, one can simply just switch to consideration of (\ref{eq:defprobucor1err})). This is of course similar to saying if a fraction (say smaller than $f_{wb}$) of the inequalities may not hold then such a fraction of patterns could be incorrectly stored. One can then reformulate (\ref{eq:defprobucor1err}) so that it provides a mathematical description for such a scenario. The resulting feasibility problem one can then consider becomes
\begin{eqnarray}
& & \d_i(H_{i,:}\x-\kappa)\geq \0,1\leq i\leq m\nonumber \\
& & \sum_{i=1}^n \d_i=(1-f_{wb})m\nonumber\\
& & \d_i\in\{0,1\},1\leq i\leq m\nonumber \\
& & \|\x\|_2=1.\label{eq:feas1err}
\end{eqnarray}
Using the replica approach Gardner developed for a problem similar to this one in \cite{Gar88}, Gardner and Derrida in \cite{GarDer88} proceeded and characterized the feasibility of (\ref{eq:feas1err}). Namely, they gave a prediction for the value of the critical storage capacity $\alpha_{c,wb}$ as a function of $f_{wb}$ and $\kappa$ so that (\ref{eq:feas1err}) is feasible (as mentioned earlier, in what follows we may often refer to $\alpha_{c,wb}$ as the storage capacity of the spherical perceptron with limited errors). The prediction given in \cite{GarDer88} essentially boils down to the following two equations: first one determines $x$ as the solution of
\begin{equation}
f_{wb}=\frac{1}{\sqrt{2\pi}}\int_{-\infty}^{\kappa-x}e^{-\frac{z^2}{2}}dz.\label{eq:xfromfmingar}
\end{equation}
Then one determines a prediction for the storage capacity $\alpha_{c,wb}^{(gar)}(\kappa)$ as
\begin{equation}
\alpha_{c,wb}^{(gar)}(\kappa)=\alpha_{c,wb}^{(gar)}(\kappa,x)=\left ( \frac{1}{\sqrt{2\pi}}\int_{\kappa-x}^{\kappa}(z-\kappa)^2e^{-\frac{z^2}{2}}dz\right )^{-1}.\label{eq:alphafromxgar}
\end{equation}
Now, assuming the standard setup (where no errors are allowed) one has $f_{wb}\rightarrow 0$ which from (\ref{eq:xfromfmingar}) implies $x\rightarrow \infty$. One then from (\ref{eq:alphafromxgar}) has
\begin{equation}
\alpha_{c,wb}^{(gar)}(\kappa,\infty)\rightarrow\left ( \frac{1}{\sqrt{2\pi}}\int_{-\infty}^{\kappa}(z-\kappa)^2e^{-\frac{z^2}{2}}dz\right )^{-1}=\left ( \frac{1}{\sqrt{2\pi}}\int_{-\kappa}^{\infty}(z+\kappa)^2e^{-\frac{z^2}{2}}dz\right )^{-1}=f_{gar}(\kappa)=\alpha_c(\kappa).\label{eq:errnoerr}
\end{equation}
In other words, if no errors are allowed (\ref{eq:xfromfmingar}) and (\ref{eq:alphafromxgar}) give the same result for the storage capacity as does (\ref{eq:garstorcap}). Now, looking back at what was presented in Figure \ref{fig:alfackappa}, one should note that when $\kappa\geq 0$ (the case primarily of interest here) the curve denotes the exact values of the storage capacity for any $\kappa$. On the other hand, one from the same plot has that if a pair $(\alpha,\kappa)$ is above the curve the memory is not stable, i.e. it is with overwhelming probability that one can not find a spherical $\x$ such that (\ref{eq:defprobucor1}) is feasible. However, if one attempts to be a bit more precise with respect to this instability one may find it useful to introduce a number of allowed wrong patterns (bits). This is in essence what (\ref{eq:xfromfmingar}) and (\ref{eq:alphafromxgar}) do. They basically attempt to characterize the number of incorrectly stored patterns when $\kappa\geq 0$ and a pair $(\alpha,\kappa)$ is above the curve given in Figure \ref{fig:alfackappa} (in fact one can use them to give a prediction for the number of incorrectly stored patterns (say $f_{wb}m$) even when $\kappa<0$). Alternatively, as framed above, one can think of all of this as a way of finding the storage capacity if a fraction of errors (incorrectly stored patterns), say $f_{wb}$ is allowed. This is of course exactly the problem that we will be attacking below and based on the above is exactly what (\ref{eq:xfromfmingar}) and (\ref{eq:alphafromxgar}) characterize.

Before proceeding further we should provide a few comments as for the potential accuracy of the above predictions. As is now well known if $\kappa\geq 0$ and $f_{wb}\rightarrow 0$ then the above prediction boils down to the standard storage capacity of the positive spherical perceptron which is based on \cite{SchTir02,SchTir03} (and later on \cite{TalBook,StojnicGardGen13}) known to be correct. On the other hand, as discussed in \cite{StojnicGardSphNeg13} (and briefly in the previous section), the above prediction is only a rigorous upper bound on the storage capacity of the negative spherical perceptron.  In fact, many of the conclusions already made in \cite{Gar88,GarDer88} indicated this kind of behavior. Namely, a stability analysis of the replica approach done in \cite{GarDer88} indicated that some of the predictions (essentially in a certain range of $(\alpha,\kappa)$ plane) related to the storage capacities when the errors are allowed may not be accurate. In \cite{Bouten94} the replica stability range given in \cite{GarDer88} was corrected a bit and as a consequence \cite{Bouten94} actually established that the replica analysis of \cite{GarDer88} may in fact produce incorrect results in the entire regime above the curve given in Figure \ref{fig:alfackappa}. Still, even if the results given in (\ref{eq:xfromfmingar}) and (\ref{eq:alphafromxgar}) are to be incorrect, they may be a fairly good approximate predictions for the storage capacity (or alternatively the fraction of incorrectly stored patterns) or they may even be say rigorous bounds on the true values (as were the predictions of \cite{Gar88} related to the negative spherical perceptron). Below we will show that the above given predictions (namely, those given in (\ref{eq:xfromfmingar}) and (\ref{eq:alphafromxgar})) are in fact rigorous upper bounds on the storage capacity of the spherical perceptron when a fraction of incorrectly stored patterns is allowed.

\section{Upper bounds on the storage capacity of the spherical perceptrons with limited errors}
\label{sec:sphpererrorsrig}

As we have mentioned at the end of the previous section, in this section we will create a set of results that will essentially establish the predictions obtained in \cite{GarDer88} (and given in (\ref{eq:xfromfmingar}) and (\ref{eq:alphafromxgar})) as rigorous upper bounds on the storage capacity of the spherical perceptron with limited errors. We start by writing an analogue to for the feasibility problem of interest here, namely the one given in (\ref{eq:feas1err})
\begin{eqnarray}
\xi_{wb}=\min_{\x,\d} \max_{\lambda\geq 0} & &  \kappa\lambda^T\diag(\d)\1- \lambda^T\diag(\d) H\x \nonumber \\
\mbox{subject to} & & \|\lambda\|_2= 1\nonumber \\
& & \sum_{i=1}^n \d_i=(1-f_{wb})m\nonumber\\
& & \d_i\in\{0,1\},1\leq i\leq m\nonumber \\
& & \|\x\|_2=1.\label{eq:uncorminmaxerr}
\end{eqnarray}
Although it is probably obvious, we mention that $\diag(\d)$ is an $m\times m$ matrix with elements of vector $\d$ on its main diagonal and zeros elsewhere. Clearly, following the logic we presented in previous sections, the sign of $\xi_{wb}$ determines the feasibility of (\ref{eq:feas1err}). In particular, if $\xi_{wb}>0$ then (\ref{eq:feas1err}) is infeasible. Given the random structure of the problem (we recall that $H$ is random) one can then pose the following probabilistic feasibility question: how small can $m$ be so that $\xi_{wb}$ in (\ref{eq:uncorminmaxerr}) is positive and (\ref{eq:feas1err}) is infeasible with overwhelming probability? In what follows we will attempt to provide an answer to such a question.

\subsection{Probabilistic analysis}
\label{sec:probanalrig}

In this section we will present a probabilistic analysis of the above optimization problem given in (\ref{eq:uncorminmaxerr}). In a nutshell, we will provide a relation between $f_{wb}$ and $\alpha=\frac{m}{n}$ so that with overwhelming probability over $H$ $\xi_{wb}>0$. This will, of course, based on the above discussion then be enough to conclude that the problem in (\ref{eq:feas1err}) is infeasible with overwhelming probability when $f_{wb}$ and $\alpha=\frac{m}{n}$ satisfy such a relation.

The analysis that we will present below will to a degree rely on a strategy we developed in \cite{StojnicRegRndDlt10,StojnicGorEx10} and utilized in \cite{StojnicGardGen13} when studying the storage capacity of the standard spherical perceptrons. We start by recalling on a set of probabilistic results from \cite{Gordon85,Gordon88} that were used as an integral part of the strategy developed in \cite{StojnicRegRndDlt10,StojnicGorEx10,StojnicGardGen13}.
\begin{theorem}(\cite{Gordon88,Gordon85})
\label{thm:Gordonmesh1} Let $X_{ij}$ and $Y_{ij}$, $1\leq i\leq n,1\leq j\leq m$, be two centered Gaussian processes which satisfy the following inequalities for all choices of indices
\begin{enumerate}
\item $E(X_{ij}^2)=E(Y_{ij}^2)$
\item $E(X_{ij}X_{ik})\geq E(Y_{ij}Y_{ik})$
\item $E(X_{ij}X_{lk})\leq E(Y_{ij}Y_{lk}), i\neq l$.
\end{enumerate}
Then
\begin{equation*}
P(\bigcap_{i}\bigcup_{j}(X_{ij}\geq \lambda_{ij}))\leq P(\bigcap_{i}\bigcup_{j}(Y_{ij}\geq \lambda_{ij})).
\end{equation*}
\end{theorem}
The following, more simpler, version of the above theorem relates to the expected values.
\begin{theorem}(\cite{Gordon85,Gordon88})
\label{thm:Gordonmesh2} Let $X_{ij}$ and $Y_{ij}$, $1\leq i\leq n,1\leq j\leq m$, be two centered Gaussian processes which satisfy the following inequalities for all choices of indices
\begin{enumerate}
\item $E(X_{ij}^2)=E(Y_{ij}^2)$
\item $E(X_{ij}X_{ik})\geq E(Y_{ij}Y_{ik})$
\item $E(X_{ij}X_{lk})\leq E(Y_{ij}Y_{lk}), i\neq l$.
\end{enumerate}
Then
\begin{equation*}
E(\min_{i}\max_{j}(X_{ij}))\leq E(\min_i\max_j(Y_{ij})).
\end{equation*}
\end{theorem}

Now, since all random quantities of interest below will concentrate around its mean values it will be enough to study only their averages. However, since it will not make writing of what we intend to present in the remaining parts of this section substantially more complicated we will present a complete probabilistic treatment and will leave the studying of the expected values for the presentation that we will give in the following section where such a consideration will substantially simplify the exposition.

We will make use of Theorem \ref{thm:Gordonmesh1} through the following lemma (the lemma is an easy consequence of Theorem \ref{thm:Gordonmesh1} and in fact is fairly similar to Lemma 3.1 in \cite{Gordon88}, see also \cite{StojnicHopBnds10,StojnicGardGen13} for similar considerations).
\begin{lemma}
Let $H$ be an $m\times n$ matrix with i.i.d. standard normal components. Let $\g$ and $\h$ be $m\times 1$ and $n\times 1$ vectors, respectively, with i.i.d. standard normal components. Also, let $g$ be a standard normal random variable and let $\zeta_{\lambda,\d}$ be a function of $\x$. Then
\begin{multline}
P(\min_{\|\x\|_2=1,\1^T\d=(1-f_{wb}) m,\d_i\in\{0,1\}}\max_{\|\lambda\|_2=1,\lambda_i\geq 0}(-\lambda^T \diag(\d)H\x+g-\zeta_{\lambda,\d})\geq 0)\\\geq
P(\min_{\|\x\|_2=1,\1^T\d=(1-f_{wb}) m,\d_i\in\{0,1\}}\max_{\|\lambda\|_2=1,\lambda_i\geq 0}(\g^T \diag(\d)\lambda+\h^T\x-\zeta_{\lambda,\d})\geq 0).\label{eq:negproblemma}
\end{multline}\label{lemma:negproblemma}
\end{lemma}
\begin{proof}
The proof is basically similar to the proof of Lemma 3.1 in \cite{Gordon88} as well as to the proof of Lemma 7 in \cite{StojnicHopBnds10}. However, one has to be a bit careful about the structures of sets of allowed values for $\lambda,\x,\d$. For completeness we will sketch the core of the argument. The remaining parts follow easily as in Lemma 3.1 in \cite{Gordon88} (or as in the proof of Lemma 7 in \cite{StojnicHopBnds10}). Namely, one starts by defining processes $X_i$ and $Y_i$ in the following way
\begin{equation}
Y_{ij}=(\lambda^{(j)})^T \diag(\d^{(i)}) H\x^{(i)} + g\quad X_{ij}=\g^T\diag(\d^{(i)})\lambda^{(j)}+\h^T\x^{(i)}.\label{eq:negexplemmaproof1}
\end{equation}
Then clearly
\begin{equation}
EY_{ij}^2=EX_{ij}^2=(\lambda^{(j)})^T \diag(\d^{(i)})\diag(\d^{(i)})\lambda^{(j)}+1.\label{eq:negexplemmaproof2}
\end{equation}
One then further has
\begin{eqnarray}
EY_{ij}Y_{ik} & = & (\lambda^{(j)})^T \diag(\d^{(i)})\diag(\d^{(i)})\lambda^{(k)}(\x^{(i)})^T\x^{(i)}+1 \nonumber \\
EX_{ij}X_{ik} & = & (\lambda^{(j)})^T \diag(\d^{(i)})\diag(\d^{(i)})\lambda^{(k)}+(\x^{(i)})^T\x^{(i)},\label{eq:negexplemmaproof3}
\end{eqnarray}
and clearly
\begin{equation}
EX_{ij}X_{ik}=EY_{ij}Y_{ik}.\label{eq:negexplemmaproof31}
\end{equation}
Moreover,
\begin{eqnarray}
EY_{ij}Y_{lk} & = & (\lambda^{(j)})^T \diag(\d^{(i)})\diag(\d^{(l)})\lambda^{(k)}(\x^{(l)})^T\x^{(i)}+1 \nonumber \\
EX_{ij}X_{lk} & = & (\lambda^{(j)})^T \diag(\d^{(i)})\diag(\d^{(l)})\lambda^{(k)}+(\x^{(l)})^T\x^{(i)}.\label{eq:negexplemmaproof32}
\end{eqnarray}
And after a small algebraic transformation
\begin{eqnarray}
\hspace{-.3in}EY_{ij}Y_{lk}-EX_{ij}X_{lk} & = & (1-(\lambda^{(j)})^T \diag(\d^{(i)})\diag(\d^{(l)})\lambda^{(k)})-
(\x^{(l)})^T\x^{(i)}(1-(\lambda^{(j)})^T \diag(\d^{(i)})\diag(\d^{(l)})\lambda^{(k)}) \nonumber \\
& = & (1-(\x^{(l)})^T\x^{(i)})(1-(\lambda^{(j)})^T \diag(\d^{(i)})\diag(\d^{(l)})\lambda^{(k)})\nonumber \\
& \geq & 0.\label{eq:negexplemmaproof4}
\end{eqnarray}
Combining (\ref{eq:negexplemmaproof2}), (\ref{eq:negexplemmaproof31}), and (\ref{eq:negexplemmaproof4}) and using results of Theorem \ref{thm:Gordonmesh1} one then easily obtains (\ref{eq:negproblemma}).
\end{proof}

Let $\zeta_{\lambda,\d}=-\kappa\lambda^T\diag(\d)\1+\epsilon_{5}^{(g)}\sqrt{n}+\xi_{wb}^{(l)}$ with $\epsilon_{5}^{(g)}>0$ being an arbitrarily small constant independent of $n$. We will first look at the right-hand side of the inequality in (\ref{eq:negproblemma}). The following is then the probability of interest
\begin{equation}
P(\min_{\|\x\|_2=1,\1^T\d=(1-f_{wb}) m,\d_i\in\{0,1\}}\max_{\|\lambda\|_2=1,\lambda_i\geq 0}(\g^T\diag(\d)\lambda+\h^T\x+\kappa\lambda^T\diag(\d)\1-\epsilon_{5}^{(g)}\sqrt{n})\geq \xi_{wb}^{(l)}).\label{eq:negprobanal0}
\end{equation}
After solving the minimization over $\x$ one obtains
\begin{multline}
P(\min_{\|\x\|_2=1,\1^T\d=(1-f_{wb}) m,\d_i\in\{0,1\}}\max_{\|\lambda\|_2=1,\lambda_i\geq 0}(\g^T\diag(\d)\lambda+\h^T\x+\kappa\lambda^T\diag(\d)\1-\epsilon_{5}^{(g)}\sqrt{n})\geq \xi_{wb}^{(l)})\\=P(f_{err}^{(r)}(\kappa)-\|\h_i\|_2-\epsilon_{5}^{(g)}\sqrt{n}\geq \xi_{wb}^{(l)}),\label{eq:negprobanal1}
\end{multline}
where
\begin{equation}
f_{err}^{(r)}(\kappa)=\min_{\1^T\d=(1-f_{wb}) m,\d_i\in\{0,1\}}\max_{\|\lambda\|_2=1,\lambda_i\geq 0}(\g^T\diag(\d)\lambda+\kappa\lambda^T\diag(\d)\1).\label{eq:defrferr}
\end{equation}
Since $\h$ is a vector of $n$ i.i.d. standard normal variables it is rather trivial that
\begin{equation}
P(\|\h\|_2<(1+\epsilon_{1}^{(n)})\sqrt{n})\geq 1-e^{-\epsilon_{2}^{(n)} n},\label{eq:devh}
\end{equation}
where $\epsilon_{1}^{(n)}>0$ is an arbitrarily small constant and $\epsilon_{2}^{(n)}$ is a constant dependent on $\epsilon_{1}^{(n)}$ but independent of $n$. Along the same lines, due to the linearity of the objective function in the definition of $f_{err}^{(r)}$ and the fact that $\g$ is a vector of $m$ i.i.d. standard normals, one has
\begin{equation}
P(f_{err}^{(r)}(\kappa)>(1-\epsilon_{1}^{(m)})f_{err}(\kappa)\sqrt{n})\geq 1-e^{-\epsilon_{2}^{(m)} m},\label{eq:devg}
\end{equation}
where
\begin{equation}
f_{err}(\kappa)=\lim_{n\rightarrow \infty}\frac{Ef_{err}^{(r)}(\kappa)}{\sqrt{n}}=\lim_{n\rightarrow\infty}\frac{E\left ( \min_{\1^T\d=(1-f_{wb}) m,\d_i\in\{0,1\}}\max_{\|\lambda\|_2=1,\lambda_i\geq 0}(\g^T\diag(\d)\lambda+\kappa\lambda^T\diag(\d)\1)\right )}{\sqrt{n}},\label{eq:ferr}
\end{equation}
and $\epsilon_{1}^{(m)}>0$ is an arbitrarily small constant and analogously as above $\epsilon_{2}^{(m)}$ is a constant dependent on $\epsilon_{1}^{(m)}$ and
$f_{err}(\kappa)$ but independent of $n$. Then a combination of (\ref{eq:negprobanal1}), (\ref{eq:devh}), and (\ref{eq:devg}) gives
\begin{multline}
P(\min_{\|\x\|_2=1,\1^T\d=(1-f_{wb}) m,\d_i\in\{0,1\}}\max_{\|\lambda\|_2=1,\lambda_i\geq 0}(\g^T\diag(\d)\lambda+\h^T\x+\kappa\lambda^T\diag(\d)\1-\epsilon_{5}^{(g)}\sqrt{n})\geq \xi_{wb}^{(l)})\\
\geq
(1-e^{-\epsilon_{2}^{(m)} m})(1-e^{-\epsilon_{2}^{(n)} n})
P((1-\epsilon_{1}^{(m)})f_{err}(\kappa)\sqrt{n}-(1+\epsilon_{1}^{(n)})\sqrt{n}-\epsilon_{5}^{(g)}\sqrt{n}\geq \xi_{wb}^{(l)}).
\label{eq:negprobanal2}
\end{multline}
If
\begin{eqnarray}
& & (1-\epsilon_{1}^{(m)})f_{err}(\kappa)\sqrt{n}-(1+\epsilon_{1}^{(n)})\sqrt{n}-\epsilon_{5}^{(g)}\sqrt{n}>\xi_{wb}^{(l)}\nonumber \\
& \Leftrightarrow & (1-\epsilon_{1}^{(m)})f_{err}(\kappa)-(1+\epsilon_{1}^{(n)})-\epsilon_{5}^{(g)}>\frac{\xi_{wb}^{(l)}}{\sqrt{n}},\label{eq:negcondxipu}
\end{eqnarray}
one then has from (\ref{eq:negprobanal2})
\begin{equation}
\lim_{n\rightarrow\infty}P(\min_{\|\x\|_2=1,\1^T\d=(1-f_{wb}) m,\d_i\in\{0,1\}}\max_{\|\lambda\|_2=1,\lambda_i\geq 0}(\g^T\diag(\d)\lambda+\h^T\x+\kappa\lambda^T\diag(\d)\1-\epsilon_{5}^{(g)}\sqrt{n})\geq \xi_{wb}^{(l)})\geq 1.\label{eq:negprobanal3}
\end{equation}
To make the result in (\ref{eq:negprobanal3}) operational one needs an estimate for $f_{err}(\kappa)$. In the following subsection we will present a way that can be used to estimate $f_{err}(\kappa)$. Before doing so we will briefly take a look at the left-hand side of the inequality in (\ref{eq:negproblemma}).

The following is then the probability of interest
\begin{equation}
P(\min_{\|\x\|_2=1,\1^T\d=(1-f_{wb}) m,\d_i\in\{0,1\}}\max_{\|\lambda\|_2=1,\lambda_i\geq 0}(\kappa\lambda^T\diag(\d)\1-\lambda^T\diag(\d)H\x+g-\epsilon_{5}^{(g)}\sqrt{n}-\xi_{wb}^{(l)})\geq 0).\label{eq:leftnegprobanal0}
\end{equation}
Since $P(g\geq\epsilon_{5}^{(g)}\sqrt{n})<e^{-\epsilon_{6}^{(g)} n}$ (where $\epsilon_{6}^{(g)}$ is, as all other $\epsilon$'s in this paper are, independent of $n$) from (\ref{eq:leftnegprobanal0}) we have
\begin{multline}
P(\min_{\|\x\|_2=1,\1^T\d=(1-f_{wb}) m,\d_i\in\{0,1\}}\max_{\|\lambda\|_2=1,\lambda_i\geq 0}(\kappa\lambda^T\diag(\d)\1-\lambda^T\diag(\d)H\x+g-\epsilon_{5}^{(g)}\sqrt{n}-\xi_{wb}^{(l)})\geq 0)
\\\leq P(\min_{\|\x\|_2=1,\1^T\d=(1-f_{wb}) m,\d_i\in\{0,1\}}\max_{\|\lambda\|_2=1,\lambda_i\geq 0}(\kappa\lambda^T\diag(\d)\1-\lambda^T\diag(\d)H\x-\xi_{wb}^{(l)})\geq 0)+e^{-\epsilon_{6}^{(g)} n}.\label{eq:leftnegprobanal1}
\end{multline}
When $n$ is large from (\ref{eq:leftnegprobanal1}) we then have
\begin{multline}
\hspace{-.7in}\lim_{n\rightarrow \infty}P(\min_{\|\x\|_2=1,\1^T\d=(1-f_{wb}) m,\d_i\in\{0,1\}}\max_{\|\lambda\|_2=1,\lambda_i\geq 0}(\kappa\lambda^T\diag(\d)\1-\lambda^T\diag(\d)H\x+g-\epsilon_{5}^{(g)}\sqrt{n}-\xi_{wb}^{(l)})\geq 0)
\\\leq  \lim_{n\rightarrow\infty}P(\min_{\|\x\|_2=1,\1^T\d=(1-f_{wb}) m,\d_i\in\{0,1\}}\max_{\|\lambda\|_2=1,\lambda_i\geq 0}(\kappa\lambda^T\diag(\d)\1-\lambda^T\diag(\d)H\x-\xi_{wb}^{(l)})\geq 0)\\
 =  \lim_{n\rightarrow\infty}P(\min_{\|\x\|_2=1}\max_{\|\lambda\|_2=1,\lambda_i\geq 0}(\kappa\lambda^T\diag(\d)\1-\lambda^T\diag(\d)H\x)\geq \xi_{wb}^{(l)}).\label{eq:leftnegprobanal2}
\end{multline}
Assuming that (\ref{eq:negcondxipu}) holds, then a combination of (\ref{eq:negproblemma}), (\ref{eq:negprobanal3}), and (\ref{eq:leftnegprobanal2}) gives
\begin{multline}
\hspace{-.5in}\lim_{n\rightarrow\infty}P(\min_{\|\x\|_2=1,\1^T\d=(1-f_{wb}) m,\d_i\in\{0,1\}}\max_{\|\lambda\|_2=1,\lambda_i\geq 0}(\kappa\lambda^T\diag(\d)\1-\lambda^T\diag(\d)H\x)\geq \xi_{wb}^{(l)})\\\geq \lim_{n\rightarrow\infty}P(\min_{\|\x\|_2=1,\1^T\d=(1-f_{wb}) m,\d_i\in\{0,1\}}\max_{\|\lambda\|_2=1,\lambda_i\geq 0}(\g^T\diag(\d)\lambda+\h^T\x+\kappa\lambda^T\diag(\d)\1-\epsilon_{5}^{(g)}\sqrt{n})\geq \xi_{wb}^{(l)})\geq 1.\label{eq:leftnegprobanal3}
\end{multline}
Of course, to have (\ref{eq:negcondxipu}) to hold and consequently to be able to use (\ref{eq:leftnegprobanal3}) one needs an estimate of $f_{err}(\kappa)$. As mentioned above, in the following subsection we will present a way that can be used to estimate $f_{err}(\kappa)$. Also, it is relatively easy to observe from the previous derivation that a lower bound on $f_{err}(\kappa)$ is sufficient. We will in fact present a way to determine a lower bound on $f_{err}(\kappa)$ (while we will not prove it, we do mention that the way we will present in the next subsection is in fact powerful enough to actually provide a precise estimate of $f_{err}(\kappa)$).

\subsection{Estimating $f_{err}(\kappa)$}
\label{sec:estferr}

We recall from (\ref{eq:ferr}) that $f_{err}(\kappa)=\lim_{n\rightarrow\infty}\frac{Ef_{err}^{(r)}(\kappa)}{\sqrt{n}}$ and from (\ref{eq:defrferr}) that
\begin{equation}
f_{err}^{(r)}(\kappa)=\min_{\1^T\d=(1-f_{wb}) m,\d_i\in\{0,1\}}\max_{\|\lambda\|_2=1,\lambda_i\geq 0}(\g^T\diag(\d)\lambda+\kappa\lambda^T\diag(\d)\1).\label{eq:defrferr1}
\end{equation}
We will first focus on $f_{err}^{(r)}(\kappa)$, i.e. on (\ref{eq:defrferr1}). To that end we will rewrite the above optimization problem in the following way
\begin{eqnarray}
f_{err}^{(r)}(\kappa) & = & \min_{\1^T\d=(1-f_{wb}) m,\d_i\in\{0,1\}}\max_{\|\lambda\|_2=1,\lambda_i\geq 0}(\g^T\diag(\d)\lambda+\kappa\lambda^T\diag(\d)\1)\nonumber \\
& = & \min_{\1^T\d=(1-f_{wb}) m,\d_i\in\{0,1\}}\|(\diag(\d)(\g+\kappa\1))_+\|_2\nonumber \\
& = & \sqrt{\min_{\1^T\d=(1-f_{wb}) m,\d_i\in\{0,1\}}\|(\diag(\d)(\g+\kappa\1))_+\|_2^2},\label{eq:defrferr2}
\end{eqnarray}
where $(\diag(\d)(\g+\kappa\1))_+$ is vector $(\diag(\d)(\g+\kappa\1))$ with negative components replaced by zeros.
After a few additional algebraic transformations we have
\begin{eqnarray}
f_{err}^{(r)}(\kappa)
& = & \sqrt{\min_{\1^T\d=(1-f_{wb}) m,\d_i\in\{0,1\}}\|(\diag(\d)(\g+\kappa\1))_+\|_2^2}\nonumber \\
& = & \sqrt{\min_{\1^T\d=(1-f_{wb}) m,\d_i\in\{0,1\}}\max_{\nu_{wb}\geq 0}(\|(\diag(\d)(\g+\kappa\1))_+\|_2^2-\nu_{wb}\d^T\1+\nu_{wb}(1-f_{wb})m)}\nonumber \\
& \geq & \sqrt{\max_{\nu_{wb}\geq 0}\min_{\1^T\d=(1-f_{wb}) m,\d_i\in\{0,1\}}(\|(\diag(\d)(\g+\kappa\1))_+\|_2^2-\nu_{wb}\d^T\1+\nu_{wb}(1-f_{wb})m)}\nonumber \\
& = & \sqrt{\max_{\nu_{wb}\geq 0}\left (\sum_{i=1}^{m}\left (\min(0,(\g_i+\kappa)_+^2-\nu_{wb})\right )+\nu_{wb}(1-f_{wb})m\right )}\nonumber \\
& = & \sqrt{\max_{\nu_{wb}\geq 0}\left (\sum_{i=1}^{m}\left (\min(0,\max((\g_i+\kappa),0)^2-\nu_{wb})\right )+\nu_{wb}(1-f_{wb})m\right )}\nonumber \\
.\label{eq:defrferr3}
\end{eqnarray}
From (\ref{eq:defrferr3}) we further have
\begin{eqnarray}
\hspace{-.3in}f_{err}(\kappa)=\lim_{n\rightarrow\infty}
\frac{Ef_{err}^{(r)}(\kappa)}{\sqrt{n}}
  & \geq & \lim_{n\rightarrow\infty}
\frac{\sqrt{\max_{\nu_{wb}\geq 0}\left (\sum_{i=1}^{m}\left (E\min(0,\max((\g_i+\kappa),0)^2-\nu_{wb})\right )+\nu_{wb}(1-f_{wb})m\right )}}{\sqrt{n}}\nonumber \\
  & \geq &
\sqrt{\max_{\nu_{wb}\geq 0}\left (\alpha\left (E\min(0,\max((\g_i+\kappa),0)^2-\nu_{wb})\right )+\nu_{wb}(1-f_{wb})\alpha\right )}.\label{eq:defrferr4}
\end{eqnarray}
The above result is already operational and one can use it to establish the bound on the storage capacity when a fraction of errors $f_{wb}$ is allowed. However, since the integrals are not that complicated one can be a bit more explicit (this will also be helpful in showing that the above bound indeed matches the results obtained in \cite{GarDer88}). To that end we have
\begin{eqnarray}
E\min(0,\max((\g_i+\kappa),0)^2-\nu_{wb}) &  = & -\nu_{wb}\int_{-\infty}^{\sqrt{\nu_{wb}}-\kappa}\frac{e^{-\frac{\g_i^2}{2}}d\g_i}{\sqrt{2\pi}}+
\frac{1}{\sqrt{2\pi}}\int_{-\kappa}^{\sqrt{\nu_{wb}}-\kappa}(\g_i+\kappa)^2e^{-\frac{\g_i^2}{2}}d\g_i\nonumber\\
& = &-\frac{\nu_{wb}}{2}+\frac{\nu_{wb}}{2}\mbox{erf}(-\frac{\sqrt{\nu_{wb}}-\kappa}{\sqrt{2}})+
\frac{1}{\sqrt{2\pi}}\int_{-\kappa}^{\sqrt{\nu_{wb}}-\kappa}(\g_i+\kappa)^2e^{-\frac{\g_i^2}{2}}d\g_i.\nonumber \\\label{eq:defrferr5}
\end{eqnarray}
To optimize over $\nu_{wb}$ we take the derivative
\begin{multline}
\frac{d(E\min(0,\max((\g_i+\kappa),0)^2-\nu_{wb})+\nu_{wb}(1-f_{wb}))}{d\nu_{wb}}  =  -\frac{1}{2}\mbox{erfc}(-\frac{\sqrt{\nu_{wb}}-\kappa}{\sqrt{2}})-\frac{\sqrt{\nu_{wb}}}{2\sqrt{2\pi}}e^{-\frac{(\sqrt{\nu_{wb}}-\kappa)^2}{2}}
\\+\frac{\sqrt{\nu_{wb}}}{2\sqrt{2\pi}}e^{-\frac{(\sqrt{\nu_{wb}}-\kappa)^2}{2}}+(1-f_{wb}).\label{eq:defrferr6}
\end{multline}
Setting the above derivative to zero gives the following condition for optimal $\nu_{wb}$, $\widehat{\nu_{wb}}$
\begin{equation}
\frac{1}{2}erfc(-\frac{\sqrt{\widehat{\nu_{wb}}}-\kappa}{\sqrt{2}})=1-f_{wb}.\label{eq:defrferr7}
\end{equation}
From (\ref{eq:defrferr7}) one then easily finds
\begin{equation}
\widehat{\nu_{wb}}=(\sqrt{2}\mbox{erfinv}(1-2f_{wb})+\kappa)^2.\label{eq:defrferr8}
\end{equation}
A combination of (\ref{eq:defrferr4}), (\ref{eq:defrferr5}), (\ref{eq:defrferr7}), and (\ref{eq:defrferr8}) then gives
\begin{eqnarray}
\hspace{-.3in}f_{err}(\kappa)=\lim_{n\rightarrow\infty}
\frac{Ef_{err}^{(r)}(\kappa)}{\sqrt{n}}
  & \geq & \lim_{n\rightarrow\infty}
\frac{\sqrt{\max_{\nu_{wb}\geq 0}\left (\sum_{i=1}^{m}\left (E\min(0,\max((\g_i+\kappa),0)^2-\nu_{wb})\right )+\nu_{wb}(1-f_{wb})m\right )}}{\sqrt{n}}\nonumber \\
  & = &
\sqrt{\alpha\frac{1}{\sqrt{2\pi}}\int_{-\kappa}^{\sqrt{\widehat{\nu_{wb}}}-\kappa}(\g_i+\kappa)^2e^{-\frac{\g_i^2}{2}}d\g_i}=\sqrt{\alpha\widehat{f_{err}}(\kappa)}.\label{eq:defrferr9}
\end{eqnarray}
Roughly speaking, if $\xi_{wb}^{(l)}$ is such that (\ref{eq:negcondxipu}) holds with $f_{err}(\kappa)$ replaced by the quantity on the right-hand side of the second equality in (\ref{eq:defrferr9}) then (\ref{eq:negprobanal3}) holds as well. This then establishes a probabilistic lower bound on $\xi_{wb}$ and as long as such lower bound is positive the optimization problem in (\ref{eq:defprobucor1}) will be infeasible. Equaling such a lower bound with zero then gives the condition to compute an upper bound on the storage capacity when a fraction no larger than $f_{wb}$ of incorrectly stored patterns is allowed. Also, while for our purposes here all inequalities in this subsection are sufficient, we mention without proving that it is actually true that they all can be replaced with equalities.

We summarize the above results in the following theorem.

\begin{theorem}
Let $H$ be an $m\times n$ matrix with i.i.d. standard normal components. Let $n$ be large and let $m=\alpha n$, where $\alpha>0$ is a constant independent of $n$. Let $\xi_{wb}$ be as in (\ref{eq:uncorminmaxerr}) and let $\kappa$ be a scalar constant independent of $n$. Let all $\epsilon$'s be arbitrarily small constants independent of $n$. Further, let $\g_i$ be a standard normal random variable and set
\begin{eqnarray}
\widehat{\nu_{wb}} & = & (\sqrt{2}\mbox{erfinv}(1-2f_{wb})+\kappa)^2\nonumber \\
\widehat{f_{err}}(\kappa) & = & \frac{1}{\sqrt{2\pi}}\int_{-\kappa}^{\sqrt{\widehat{\nu_{wb}}}-\kappa}(\g_i+\kappa)^2e^{-\frac{\g_i^2}{2}}d\g_i.\label{eq:thmerrc30}
\end{eqnarray}
Let $\xi_{wb}^{(l)}$ be a scalar such that
\begin{equation}
(1-\epsilon_{1}^{(m)})\sqrt{\alpha \widehat{f_{err}}(\kappa)}-(1+\epsilon_{1}^{(n)})-\epsilon_{5}^{(g)}  >  \frac{\xi_{wb}^{(l)}}{\sqrt{n}}.\label{eq:condxinthmstoc30}
\end{equation}
Then
\begin{equation}
\hspace{-.3in} \lim_{n\rightarrow\infty}P(\xi_{wb}\geq \xi_{wb}^{(l)})=\lim_{n\rightarrow\infty}P(\min_{\|\x\|_2=1,\1^T\d=(1-f_{wb}) m,\d_i\in\{0,1\}}\max_{\|\lambda\|_2=1,\lambda_i\geq 0}(\kappa\lambda^T\diag(\d)\1-\lambda^T\diag(\d)H\x)\geq \xi_{wb}^{(l)})\geq 1. \label{eq:probthmstoc30poskappa}
\end{equation}
\label{thm:Stoc30err}
\end{theorem}
\begin{proof}
Follows from the discussion presented above.
\end{proof}

In a more informal language (as earlier, essentially ignoring all technicalities and $\epsilon$'s) one has that as long as
\begin{equation}
\alpha>\frac{1}{\widehat{f_{err}}(\kappa)},\label{eq:condalphauncorlberr}
\end{equation}
the problem in (\ref{eq:feas1err}) will be infeasible with overwhelming probability. It is an easy exercise to show that the right hand side of (\ref{eq:condalphauncorlberr}) matches the right-hand side of (\ref{eq:alphafromxgar}) (one should simply think of $x$ in (\ref{eq:alphafromxgar}) as $\sqrt{\widehat{\nu_{wb}}}$). This is then enough to conclude that the prediction for the storage capacity given in \cite{GarDer88} for the case when a fraction of errors no larger than $f_{wb}$ is allowed is in fact a rigorous upper bound on the true value of the corresponding storage capacity.

The results obtained based on the above theorem as well as those predicted based on the replica theory and given in (\ref{eq:alphafromxgar}) (and of course in \cite{GarDer88}) are presented in Figure \ref{fig:fminalphacc30}. To be in a complete agreement with what was done in \cite{GarDer88} we selected three different cases for $\kappa$, namely, $\kappa\in\{0,0.5,1\}$ (these are of course the same cases selected in \cite{GarDer88}). For the values of $\alpha$ that are to the right of the given curve the memory will have more than $f_{wb}m$ incorrectly stored patterns with overwhelming probability. Also, we do mention without going into further details that one can create similar curves for negative $\kappa$ as well. While the corresponding mathematical problems are very interesting, we chose to present only the positive $\kappa$ results. This is mainly done because the negative $\kappa$ case may not be of primary interest in the context of neural networks and storage capacities of memories induced by them. As a consequence we will present the discussion in this direction elsewhere.

\begin{figure}[htb]
\centering
\centerline{\epsfig{figure=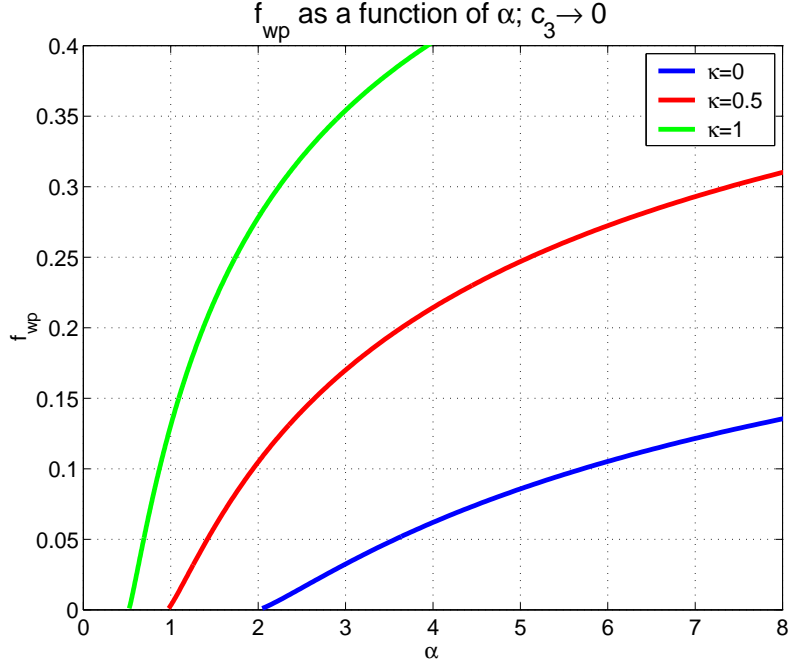,width=10.5cm,height=9cm}}
\caption{$f_{wb}$ as a function of $\alpha$ (or alternatively, $\alpha$ as a function of $f_{wb}$) for $\kappa\in\{0,0.5,1\}$}
\label{fig:fminalphacc30}
\end{figure}

\section{Lowering the storage capacity}
\label{sec:sphpererrorslow}

The results we presented in the previous section provide a rigorous upper bound on the storage capacity of the spherical perceptron when a fraction $f_{wb}$ of stored patterns is allowed to be stored erroneously. Given known results for the storage capacity of the spherical perceptron in the standard case (i.e. when $f_{wb}\rightarrow 0$) one may be tempted to believe that some of the above results are actually exact. The reasoning could be along the following lines: as shown in \cite{SchTir02,SchTir03} and confirmed in\cite{TalBook,StojnicGardGen13}, when $\kappa\geq 0$ the statistical mechanics predictions for the storage capacity of the standard spherical perceptron are actually correct. So, one may continue such a reasoning and predict that the statistical mechanics type of observations may be correct when $\kappa\geq 0$ even when it comes to the storage capacities when the limited errors are allowed. On the other hand, a very strong argument against such a logic would be that the limited errors introduce a combinatorial aspect to the problem at hand and the replica symmetry type of statistical mechanics approach may stop being exact. In this section we will present a collection of results that can be used to potentially lower the upper bounds for the storage capacity when the errors are allowed thereby opening an avenue for rigorously showing that the replica symmetry predictions are (as shown in the previous section) only upper bounds. In fact, a limited collection of numerical results that we will present below indicates that it may indeed be true that in certain range of problem parameters (essentially in certain range of $(\alpha,\kappa,f_{wb})$ space) the results presented in the previous section are indeed only the upper bounds.

Before proceeding further with the presentation of the strategy we believe can be used for lowering the upper bounds from the previous section, we first recall on a few technical details from previous sections that we will need here again. We start by recalling on the optimization problem that we will consider here. As is probably obvious, it is basically the one given in (\ref{eq:uncorminmaxerr})
\begin{eqnarray}
\xi_{wb}=\min_{\x,\d} \max_{\lambda\geq 0} & &  \kappa\lambda^T\diag(\d)\1- \lambda^T\diag(\d) H\x \nonumber \\
\mbox{subject to} & & \|\lambda\|_2= 1\nonumber \\
& & \sum_{i=1}^n \d_i=(1-f_{wb})m\nonumber\\
& & \d_i\in\{0,1\},1\leq i\leq m\nonumber \\
& & \|\x\|_2=1.\label{eq:uncorminmaxerrlow}
\end{eqnarray}
where $\1$ is an $m$-dimensional column vector of all $1$'s. As mentioned below (\ref{eq:uncorminmaxerr}), a probabilistic characterization of the sign of $\xi_{wb}$ would be enough to determine the storage capacity or its bounds. Below, we provide a way that can be used to probabilistically characterize $\xi_{wb}$. Moreover, as mentioned at the beginning of the previous section, since $\xi_{wb}$ will concentrate around its mean for our purposes here it will then be enough to study only its mean $E\xi_{wb}$. We do so by relying on the strategy developed in \cite{StojnicMoreSophHopBnds10} (and employed in \cite{StojnicGardSphNeg13}) and ultimately on the following set of results from \cite{Gordon85}. (The following theorem presented in \cite{StojnicMoreSophHopBnds10} is in fact a slight alternation of the original results from \cite{Gordon85}.)
\begin{theorem}(\cite{Gordon85})
\label{thm:Gordonneg1} Let $X_{ij}$ and $Y_{ij}$, $1\leq i\leq n,1\leq j\leq m$, be two centered Gaussian processes which satisfy the following inequalities for all choices of indices
\begin{enumerate}
\item $E(X_{ij}^2)=E(Y_{ij}^2)$
\item $E(X_{ij}X_{ik})\geq E(Y_{ij}Y_{ik})$
\item $E(X_{ij}X_{lk})\leq E(Y_{ij}Y_{lk}), i\neq l$.
\end{enumerate}
Let $\psi_{ij}()$ be increasing functions on the real axis. Then
\begin{equation*}
E(\min_{i}\max_{j}\psi_{ij}(X_{ij}))\leq E(\min_{i}\max_{j}\psi_{ij}(Y_{ij})).
\end{equation*}
Moreover, let $\psi_{ij}()$ be decreasing functions on the real axis. Then
\begin{equation*}
E(\max_{i}\min_{j}\psi_{ij}(X_{ij}))\geq E(\max_{i}\min_{j}\psi_{ij}(Y_{ij})).
\end{equation*}
\begin{proof}
The proof of all statements but the last one is of course given in \cite{Gordon85}. The proof of the last statement trivially follows and in a slightly different scenario is given for completeness in \cite{StojnicMoreSophHopBnds10}.
\end{proof}
\end{theorem}
The strategy that we will present below will utilize the above theorem to lift the above mentioned lower bound on $\xi_{wb}$ (of course since we talk in probabilistic terms, under bound on $\xi_{wb}$ we essentially assume a bound on $E\xi_{wb}$). We do mention that in Section \ref{sec:sphpererrorsrig} we relied on a variant of the above theorem to create a probabilistic lower bound on $\xi_{wb}$. However, the strategy employed in Section \ref{sec:sphpererrorsrig} relied only on a basic version of the above theorem which assumes $\psi_{ij}(x)=x$. Here, we will substantially upgrade the strategy from Section \ref{sec:sphpererrorsrig} by looking at a very simple (but way better) different version of $\psi_{ij}()$.

\subsection{Lifting lower bound on $\xi_{wb}$}
\label{sec:uncorgardlb}

In \cite{StojnicMoreSophHopBnds10,StojnicGardSphNeg13} we established lemmas very similar to the following one:
\begin{lemma}
Let $A$ be an $m\times n$ matrix with i.i.d. standard normal components. Let $\g$ and $\h$ be $m\times 1$ and $n\times 1$ vectors, respectively, with i.i.d. standard normal components. Also, let $g$ be a standard normal random variable and let $c_3$ be a positive constant. Then
\begin{multline}
E(\max_{\|\x\|_2=1, \1^T\d=(1-f_{wb}) m,\d_i\in\{0,1\}}\min_{\|\lambda\|_2=1,\lambda_i\geq 0}e^{-c_3(-\lambda^T\diag(\d) H\x + g +\kappa\lambda^T\diag(\d)\1)})\\\leq E(\max_{\|\x\|_2=1, \1^T\d=(1-f_{wb}) m,\d_i\in\{0,1\}}\min_{\|\lambda\|_2=1,\lambda_1\geq 0}e^{-c_3(\g^T\diag(\d)\lambda+\h^T\x+\kappa\lambda^T\diag(\d)\1)}).\label{eq:negexplemmalow}
\end{multline}\label{lemma:negexplemmalow}
\end{lemma}
\begin{proof}
The proof is a combination of Theorem \ref{thm:Gordonneg1} and the proof of Lemma \ref{lemma:negproblemma}. We will omit the details as they are pretty much the same as the those in the proof of Lemma \ref{lemma:negproblemma} and the corresponding lemmas in \cite{StojnicMoreSophHopBnds10,StojnicGardSphNeg13}. However, we do mention that the main difference between this lemma and the corresponding ones in \cite{StojnicMoreSophHopBnds10,StojnicGardSphNeg13} is in the structure of the sets of allowed values for $\x$, $\d$, and $\lambda$. However, such a difference introduces no structural changes in the proof.
\end{proof}

Following step by step what was done after Lemma 3 in \cite{StojnicMoreSophHopBnds10} one arrives at the following analogue of \cite{StojnicMoreSophHopBnds10}'s equation $(57)$:
\begin{multline}
E(\min_{\|\x\|_2=1, \1^T\d=(1-f_{wb}) m,\d_i\in\{0,1\}}\max_{\|\lambda\|_2=1,\lambda_i\geq 0}(-\lambda^T \diag(\d)H\x+\kappa\lambda^T\diag(\d)\1))\\\hspace{-.3in}\geq
\frac{c_3}{2}-\frac{1}{c_3}\log(E(\max_{\|\x\|_2=1}(e^{-c_3\h^T\x})))
-\frac{1}{c_3}\log(E(\max_{\1^T\d=(1-f_{wb}) m,\d_i\in\{0,1\}}\min_{\|\lambda\|_2=1,\lambda_i\geq 0}(e^{-c_3(\g^T\diag(\d)\lambda+\kappa\lambda^T\diag(\d)\1)}))).\\\label{eq:chneg8}
\end{multline}
Let $c_3=c_3^{(s)}\sqrt{n}$ where $c_3^{(s)}$ is a constant independent of $n$. Then (\ref{eq:chneg8}) becomes
\begin{multline}
\frac{E(\min_{\|\x\|_2=1, \1^T\d=(1-f_{wb}) m,\d_i\in\{0,1\}}\max_{\|\lambda\|_2=1,\lambda_i\geq 0}(-\lambda^T \diag(\d)H\x+\kappa\lambda^T\diag(\d)\1))}{\sqrt{n}}
\\\geq
\frac{c_3^{(s)}}{2}-\frac{1}{nc_3^{(s)}}\log(E(\max_{\|\x\|_2=1}(e^{-c_3^{(s)}\sqrt{n}\h^T\x}))) \\
 -  \frac{1}{nc_3^{(s)}}\log(E(\max_{\1^T\d=(1-f_{wb}) m,\d_i\in\{0,1\}}\min_{\|\lambda\|_2=1,\lambda_i\geq 0}(e^{-c_3^{(s)}\sqrt{n}(\g^T\diag(\d)\lambda+\kappa\lambda^T\diag(\d)\1)}))) \\
 = -(-\frac{c_3^{(s)}}{2}+I_{sph}(c_3^{(s)})+I_{wb}(c_3^{(s)},\alpha,\kappa,f_{wb})),\label{eq:chneg9}
\end{multline}
where
\begin{eqnarray}
I_{sph}(c_3^{(s)}) & = & \frac{1}{nc_3^{(s)}}\log(E(\max_{\|\x\|_2=1}(e^{-c_3^{(s)}\sqrt{n}\h^T\x})))\nonumber \\
I_{wb}(c_3^{(s)},\alpha,\kappa,f_{wb}) & = & \frac{1}{nc_3^{(s)}}\log(E(\max_{\1^T\d=(1-f_{wb}) m,\d_i\in\{0,1\}}\min_{\|\lambda\|_2=1,\lambda_i\geq 0}(e^{-c_3^{(s)}\sqrt{n}(\g^T\diag(\d)\lambda+\kappa\lambda^T\diag(\d)\1)}))).\nonumber \\\label{eq:defIs}
\end{eqnarray}
Moreover, \cite{StojnicMoreSophHopBnds10} also established
\begin{equation}
\hspace{-.5in}I_{sph}(c_3^{(s)}) = \frac{1}{nc_3^{(s)}}\log(E(\max_{\|\x\|_2=1}(e^{-c_3^{(s)}\sqrt{n}\h^T\x})))
 \doteq\widehat{\gamma^{(s)}}-\frac{1}{2c_3^{(s)}}\log(1-\frac{c_3^{(s)}}{2\widehat{\gamma^{(s)}}}),\label{eq:ubmorsoph}
\end{equation}
where
\begin{equation}
\widehat{\gamma^{(s)}}=\frac{2c_3^{(s)}+\sqrt{4(c_3^{(s)})^2+16}}{8},\label{eq:gamaiden3}
\end{equation}
and $\doteq$ stands for equality when $n\rightarrow\infty$ (as mentioned in \cite{StojnicMoreSophHopBnds10}, $\doteq$ in (\ref{eq:ubmorsoph}) is exactly what was shown in \cite{SPH}.

To be able to use the bound in (\ref{eq:chneg8}) we would also need a characterization of $I_{wb}(c_3^{(s)},\alpha,\kappa,f_{wb})$. Below we provide such a characterization. We start with the following observation that easily follows from (\ref{eq:defrferr1})
\begin{eqnarray}
I_{wb}(c_3^{(s)},\alpha,\kappa,f_{wb}) & = & \frac{1}{nc_3^{(s)}}\log(E(\max_{\1^T\d=(1-f_{wb}) m,\d_i\in\{0,1\}}\min_{\|\lambda\|_2=1,\lambda_i\geq 0}(e^{-c_3^{(s)}\sqrt{n}(\g^T\diag(\d)\lambda+\kappa\lambda^T\diag(\d)\1)})))\nonumber \\
& = & \frac{1}{nc_3^{(s)}}\log(E(e^{-c_3^{(s)}\sqrt{n}f_{err}^{(r)}(\kappa)}))
\label{eq:Iwb1}
\end{eqnarray}
From (\ref{eq:defrferr3}) one then has
\begin{eqnarray}
f_{err}^{(r)}(\kappa)
& \geq & \sqrt{\max_{\nu_{wb}\geq 0}\left (\sum_{i=1}^{m}\left (\min(0,\max((\g_i+\kappa),0)^2-\nu_{wb})\right )+\nu_{wb}(1-f_{wb})m\right )}\nonumber \\
& = & \min_{\gamma_{wb}\geq 0}\left (\frac{\max_{\nu_{wb}\geq 0}\left (\sum_{i=1}^{m}\left (\min(0,\max((\g_i+\kappa),0)^2-\nu_{wb})\right )+\nu_{wb}(1-f_{wb})m\right )}{4\gamma_{wb}}+ \gamma_{wb}\right )\nonumber \\
& = & \min_{\gamma_{wb}\geq 0}\max_{\nu_{wb}\geq 0}\left (\frac{f_1(\nu_{wb},\g,\kappa,f_{wb})}{4\gamma_{wb}}+ \gamma_{wb}\right ),
\label{eq:Iwb2}
\end{eqnarray}
where
\begin{equation}
f_1(\nu_{wb},\g,\kappa,f_{wb})=\left (\sum_{i=1}^{m}\left (\min(0,\max((\g_i+\kappa),0)^2-\nu_{wb})\right )+\nu_{wb}(1-f_{wb})m\right ).\label{eq:deff1}
\end{equation}
Then a combination of (\ref{eq:Iwb1}), (\ref{eq:Iwb2}), and (\ref{eq:deff1}) gives
\begin{multline}
\hspace{-.5in}I_{wb}(c_3^{(s)},\alpha,\kappa,f_{wb})
 =  \frac{1}{nc_3^{(s)}}\log(E(e^{-c_3^{(s)}\sqrt{n}f_{err}^{(r)}(\kappa)}))\leq
 \frac{1}{nc_3^{(s)}}\log(E(e^{-c_3^{(s)}\sqrt{n}\min_{\gamma_{wb}\geq 0}\max_{\nu_{wb}\geq 0}\left (\frac{f_1(\nu_{wb},\g,\kappa,f_{wb})}{4\gamma_{wb}}+ \gamma_{wb}\right )}))\\
\hspace{-.5in}\doteq  \frac{1}{nc_3^{(s)}}\min_{\gamma_{wb}\geq 0}\max_{\nu_{wb}\geq 0}\log(E(e^{-c_3^{(s)}\sqrt{n}\left (\frac{f_1(\nu_{wb},\g,\kappa,f_{wb})}{4\gamma_{wb}}+ \gamma_{wb}\right )}))=
\min_{\gamma_{wb}\geq 0}\max_{\nu_{wb}\geq 0}(-\frac{\gamma_{wb}}{\sqrt{n}}+\frac{1}{nc_3^{(s)}}\log(Ee^{-c_3^{(s)}\sqrt{n}(\frac{f_1(\nu_{wb},\g,\kappa,f_{wb})}{4\gamma_{wb}})}))\\
=\min_{\gamma_{wb}\geq 0}\max_{\nu_{wb}\geq 0}(-\frac{\alpha\nu_{wb}(1-f_{wb})\sqrt{n}}{4\gamma_{wb}}-\frac{\gamma_{wb}}{\sqrt{n}}+
\frac{\alpha}{c_3^{(s)}}\log(Ee^{-c_3^{(s)}\sqrt{n}(\frac{\left (\min(0,\max((\g_i+\kappa),0)^2-\nu_{wb})\right )}{4\gamma_{wb}})}))\\
,\label{eq:gamaiden1sec}
\end{multline}
where $\doteq$ denotes an equality as $n\rightarrow \infty$ and follows based on considerations similar to those presented in \cite{SPH} (and discussed in \cite{StojnicMoreSophHopBnds10}). Since here the things may appear seemingly more involved (than say in \cite{StojnicMoreSophHopBnds10})), one can adopt a simpler way of reasoning: namely, since the above inequalities hold for any $\nu_{wb}\geq 0$ one can fix one of them (say, exactly the one that we will later on determine as the optimal one) and then apply the mechanism from \cite{SPH} only to the minimization over $\gamma_{wb}\geq 0$ which is exactly what is the done in \cite{SPH} and references recalled on therein. In that case $\doteq$ should be replaced with an $\leq $ inequality which is enough for our purposes here (however, tightening over $\nu_{wb}$ would be enough to obtain the above mentioned limiting equality, i.e. $\doteq$). Moreover, following what we mentioned in the previous section, the above inequalities can in fact be shown to be equalities in $n\rightarrow \infty$ limit, since the inequality in (\ref{eq:Iwb2}) can in fact be replaced by an equality as well. As mentioned in the previous section though, we skip showing this as it is not really needed for the results that we will present below (showing this is not really difficult but in our opinion diverts attention from the final results which it improves in no way).

Now if one sets $\gamma_{wb}=\gamma_{wb}^{(s)}\sqrt{n}$ then (\ref{eq:gamaiden1sec}) gives
\begin{eqnarray}
\hspace{-.3in}I_{wb}(c_3^{(s)},\alpha,\kappa,f_{wb}) & \leq &  \min_{\gamma_{wb}\geq 0}\max_{\nu_{wb}\geq 0}(-\frac{\alpha\nu_{wb}(1-f_{wb})\sqrt{n}}{4\gamma_{wb}}-\frac{\gamma_{wb}}{\sqrt{n}}+
\frac{\alpha}{c_3^{(s)}}\log(Ee^{-c_3^{(s)}\sqrt{n}(\frac{\left (\min(0,\max((\g_i+\kappa),0)^2-\nu_{wb})\right )}{4\gamma_{wb}})}))\nonumber \\
& = &  \min_{\gamma_{wb}^{(s)}\geq 0}\max_{\nu_{wb}\geq 0}(-\frac{\alpha\nu_{wb}(1-f_{wb})}{4\gamma_{wb}^{(s)}}-\gamma_{wb}^{(s)}+
\frac{\alpha}{c_3^{(s)}}\log(I_{wb}^{(1)}(c_3^{(s)},\gamma_{per}^{(s)},\nu_{wb},\kappa))),\label{eq:gamaiden2sec}
\end{eqnarray}
where
\begin{equation}
I_{wb}^{(1)}(c_3^{(s)},\gamma_{per}^{(s)},\nu_{wb},\kappa) = Ee^{-c_3^{(s)}\frac{\left (\min(0,\max((\g_i+\kappa),0)^2-\nu_{wb})\right )}{4\gamma_{wb}^{(s)}}}.\label{eq:defI1I2sec}
\end{equation}
A combination of (\ref{eq:gamaiden2sec}) and (\ref{eq:defI1I2sec}) is then enough to enable us to use the bound in (\ref{eq:chneg8}). However, one can be a bit more explicit when it comes to $I_{wb}^{(1)}(c_3^{(s)},\gamma_{per}^{(s)},\nu_{wb},\kappa)$. Set
\begin{eqnarray}
p & = & 1+\frac{c_3^{(s)}}{2\gamma_{wb}^{(s)}}\nonumber \\
q & = & \frac{c_3^{(s)}\kappa}{2\gamma_{wb}^{(s)}}\nonumber \\
r & = & \frac{c_3^{(s)}\kappa^2}{4\gamma_{wb}^{(s)}}\nonumber \\
s_1 & = & -\kappa\sqrt{p}+\frac{q}{\sqrt{p}} \nonumber \\
s_2 & = & (\sqrt{\nu_{wb}}-\kappa)\sqrt{p}+\frac{q}{\sqrt{p}} \nonumber \\
C & = & \frac{exp\left (\frac{q^2}{2p}-r\right )exp\left (\frac{c_3^{(s)}\nu_{wb}}{4\gamma_{wb}^{(s)}}\right )}{\sqrt{p}}.\label{eq:helpdeflow}
\end{eqnarray}
Further set
\begin{eqnarray}
I_{wb}^{(1,1)}(c_3^{(s)},\gamma_{per}^{(s)},\nu_{wb},\kappa) & = & exp\left (\frac{c_3^{(s)}\nu_{wb}}{4\gamma_{wb}^{(s)}}\right )\frac{1}{2}\mbox{erfc}\left (\frac{\kappa}{\sqrt{2}}\right ) \nonumber \\
I_{wb}^{(1,2)}(c_3^{(s)},\gamma_{per}^{(s)},\nu_{wb},\kappa) & = & \frac{C}{2}\left (\mbox{erfc}\left (\frac{s_1}{\sqrt{2}}\right )-\mbox{erfc}\left (\frac{s_2}{\sqrt{2}}\right )\right ) \nonumber \\
I_{wb}^{(1,3)}(\nu_{wb},\kappa) & = & \frac{1}{2}\mbox{erfc}\left (\frac{(\sqrt{\nu_{wb}}-\kappa)}{\sqrt{2}}\right ).\label{eq:defIwb11slow}
\end{eqnarray}
Then solving the integrals in (\ref{eq:defI1I2sec}) gives
\begin{equation}
I_{wb}^{(1)}(c_3^{(s)},\gamma_{per}^{(s)},\nu_{wb},\kappa)=I_{wb}^{(1,1)}(c_3^{(s)},\gamma_{per}^{(s)},\nu_{wb},\kappa)
+I_{wb}^{(1,2)}(c_3^{(s)},\gamma_{per}^{(s)},\nu_{wb},\kappa)
+I_{wb}^{(1,3)}(\nu_{wb},\kappa).\label{eq:Iwb1finallow}
\end{equation}

We summarize the results from this section in the following theorem.

\begin{theorem}
Let $H$ be an $m\times n$ matrix with $\{-1,1\}$ i.i.d. standard normal components. Let $n$ be large and let $m=\alpha n$, where $\alpha>0$ is a constant independent of $n$. Let $\xi_{wb}$ be as in (\ref{eq:uncorminmax}) and let $\kappa<0$ be a scalar constant independent of $n$. Let all $\epsilon$'s be arbitrarily small constants independent of $n$. Further, let $\g_i$ be a standard normal random variable. Set
\begin{equation}
\widehat{\gamma^{(s)}}=\frac{2c_3^{(s)}+\sqrt{4(c_3^{(s)})^2+16}}{8},\label{eq:gamaliftthmlow}
\end{equation}
and
\begin{equation}
I_{sph}(c_3^{(s)}) = \widehat{\gamma^{(s)}}-\frac{1}{2c_3^{(s)}}\log(1-\frac{c_3^{(s)}}{2\widehat{\gamma^{(s)}}}).\label{eq:Isphthmlow}
\end{equation}
Let $I_{wb}^{(1)}(c_3^{(s)},\gamma_{wb}^{(s)},\nu_{wb},\kappa)$ be defined through (\ref{eq:helpdeflow})-(\ref{eq:Iwb1finallow}). Set
\begin{equation}
\widehat{I_{wb}}(c_3^{(s)},\alpha,\kappa,f_{wb})= \min_{\gamma_{wb}^{(s)}\geq 0}\max_{\nu_{wb}\geq 0}(-\frac{\alpha\nu_{wb}(1-f_{wb})}{4\gamma_{wb}^{(s)}}-\gamma_{wb}^{(s)}+
\frac{\alpha}{c_3^{(s)}}\log(I_{wb}^{(1)}(c_3^{(s)},\gamma_{per}^{(s)},\nu_{wb},\kappa))).\label{eq:Iperthmlow}
\end{equation}
If $\alpha$ is such that
\begin{equation}
-\widehat{\xi_{wb}^{(l,lift)}}=\min_{c_3^{(s)}\geq 0}(-\frac{c_3^{(s)}}{2}+I_{sph}(c_3^{(s)})+\widehat{I_{wb}}(c_3^{(s)},\alpha,\kappa,f_{wb}))<0,\label{eq:condliftsphnegthmlow}
\end{equation}
then (\ref{eq:feas1err}) is infeasible with overwhelming probability.
\label{thm:liftnegsphperlow}
\end{theorem}
\begin{proof}
Follows from the previous discussion by combining (\ref{eq:uncorminmaxerrlow}) and (\ref{eq:chneg9}), and  by noting that the bound given in (\ref{eq:chneg9}) holds for any $c_3^{(s)}\geq 0$ and could therefore be tightened by additionally optimizing over $c_3^{(s)}\geq 0$.
\end{proof}

The results one can obtain for the storage capacity based on the above theorem are presented in Figure \ref{fig:fminalphacoptc3}. Similarly to what we mentioned when presenting the results for the negative spherical perceptron (namely, those shown in Figure \ref{fig:liftsphneg}) the results presented in Figure \ref{fig:fminalphacoptc3} should be taken only as an illustration. They are obtained as a result of a numerical optimization. Remaining finite precision errors are of course possible and could affect the validity of the obtained results (as in Section \ref{sec:negphper} we do believe though that this is not the case). Either way, we would like to emphasize once again that the results presented in Theorem \ref{thm:liftnegsphperlow} are completely mathematically rigorous. Their representation given in Figure \ref{fig:fminalphacoptc3} may have been a bit imprecise due to numerical computations needed to obtain these plots.
\begin{figure}[htb]
\centering
\centerline{\epsfig{figure=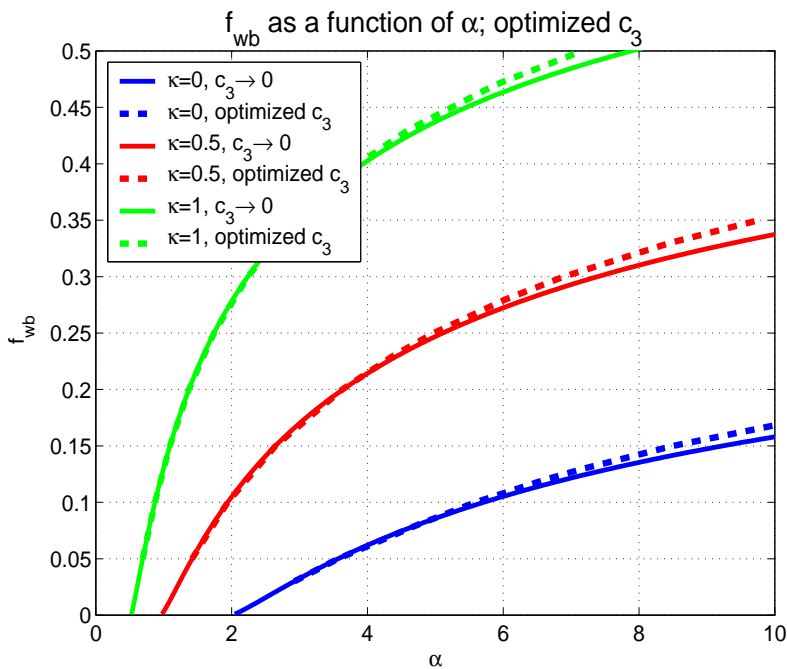,width=10.5cm,height=9cm}}
\caption{$f_{wb}$ as a function of $\alpha$ (or alternatively, $\alpha$ as a function of $f_{wb}$) for $\kappa\in\{0,0.5,1\}$; optimized $c_3^{(s)}$}
\label{fig:fminalphacoptc3}
\end{figure}

As, for the plot actually shown in Figure \ref{fig:fminalphacoptc3}, we basically showed the same type of plots we have already shown in Figure \ref{fig:fminalphacc30} (these are denoted by $c_3\rightarrow 0$ label as they rely on Theorem \ref{thm:Stoc30err} which indeed follows from Theorem \ref{thm:liftnegsphperlow} assuming that $c_3\rightarrow 0$). In addition to that we have shown what kind of effect on these plots the results of Theorem \ref{thm:liftnegsphperlow} may have (these are denoted by optimized $c_3$ label to indicate that they are obtained based on the set of results given in Theorem \ref{thm:liftnegsphperlow} which ultimately assumes an optimization over (a scaled version of) $c_3\geq 0$). To be even more specific, we showed the relation between the storage capacity $\alpha$ and the maximal fraction of allowed errors $f_{wb}$. We did so for three different values of $\kappa$, namely for $\kappa\in\{0,0.5,1\}$ and we did so for the case when $c_3^{(s)}$ in Theorem \ref{thm:liftnegsphperlow} is assumed to tend to zero and for the case when it is assumed to take the optimal value predicted by the results of Theorem \ref{thm:liftnegsphperlow}. The dotted curves indicate that an improvement in the storage capacities characterization may be possible. In other words, it is possible that in certain range of parameters $(f_{wb},\kappa)$ the storage capacity results obtained based on Theorem \ref{thm:liftnegsphperlow} may indeed be lower than those obtained based on Theorem \ref{thm:Stoc30err}. Since the results obtained in Theorem \ref{thm:Stoc30err} match the predictions obtained based on the replica approach from statistical mechanics (assuming replica symmetry) this then basically indicates that the true values of the storage capacities when the errors in stored patterns are allowed
may be lower than what replica symmetry statistical mechanics approach predicts (and what our results from Section \ref{sec:sphpererrors}, essentially Theorem \ref{thm:Stoc30err}, confirm as rigorous upper bounds).

On the other hand, while we view the presented improvement (i.e. lowering) of the storage capacity as conceptually substantial it is practically not so easily visible on the given plots. For that reason we in Tables \ref{tab:liftsphnegtab1}-\ref{tab:liftsphnegtab6} below also give the concrete values that we obtained for the storage capacities (and all optimizing parameters appearing in Theorem \ref{thm:liftnegsphperlow}). To do so we selected a range of parameters $f_{wb}$ for each of three $\kappa$'s where one starts seeing a difference in the capacity values. Also, as we have just mentioned above, since the numerical precision (especially the optimization over $\gamma_{wb}$) may have jeopardized the rigorousness of the results presented in Figure \ref{fig:fminalphacoptc3} we provide the values of all optimizing parameters (also, as mentioned earlier, we do not believe that the numerical work introduced any substantial inaccuracies).
We denote by $\alpha_{c,wb}^{(u,low)}$ the smallest $\alpha$ such that (\ref{eq:condliftsphnegthmlow}) holds. Along the same lines we denote by $\alpha_{c,wb}^{(u)}$ the smallest $\alpha$ such that (\ref{eq:condalphauncorlberr}) holds. In fact, $\alpha_{c,wb}^{(u)}$ can also be obtained from Theorem \ref{thm:Stoc30err} by taking $c_3^{(s)}\rightarrow 0$ and consequently $\gamma_{wb}^{(s)}\rightarrow\frac{1}{2}$ (of course as mentioned on numerous occasions in the previous section, such an $\alpha_{c,wb}^{(u)}$ matches $\alpha_{c,wb}^{(gar)}$ given in (\ref{eq:alphafromxgar}) and obtained in \cite{GarDer88}). Moreover, $\nu_{wb}$ obtained based on results of Theorem \ref{thm:liftnegsphperlow} with $c_3^{(s)}\rightarrow 0$ (and consequently $\gamma_{wb}^{(s)}\rightarrow\frac{1}{2}$) corresponds to $\widehat{\nu_{wb}}$ from (\ref{eq:thmerrc30}) and, as argued in the previous section, matches $x^2$ in (\ref{eq:alphafromxgar}).

\begin{table}
\caption{Lowered upper bounds on $\alpha_{c,wb}$ -- lower $\alpha,f_{wb}$ regime; $\kappa=0$, optimized parameters}\vspace{.1in}
\hspace{-0in}\centering
\begin{tabular}{||c||c|c|c|c||}\hline\hline
 $f_{wb}$                                              & $0.05$   & $0.08$   & $0.01$   & $0.12$    \\ \hline\hline
 $\widehat{\xi_{wb}^{(l,lift)}}$                       & $0.0000$ & $0.0000$ & $5.475e-07$  & $3.389e-06$  \\ \hline
 $c_{3}^{(s)}$                                         & $0.0000$ & $0.0000$ & $0.7907$  & $1.1211$   \\ \hline
 $\gamma_{wb}^{(s)}$                                   & $0.5000$ & $0.5000$ & $0.3400$  & $0.2929$   \\ \hline
 $\nu_{wb}$                                            & $2.7055$ & $1.9742$ & $1.0056$  & $0.7055$   \\ \hline
 $\alpha_{c,wb}^{(u,low)}$                             & $\mathbf{3.5669}$ & $\mathbf{4.7368}$ & $\mathbf{5.5910}$ &
 $\mathbf{6.6138}$  \\ \hline\hline
 $\nu_{wb}$ ($c_3^{(s)}\rightarrow 0$,
 $\gamma_{wb}^{(s)}\rightarrow \frac{1}{2}$)                        & $2.7055$ & $1.9742$ & $1.6423$  & $1.3806$ \\ \hline
 $\alpha_{c,wb}^{(u)}$ ($c_3^{(s)}\rightarrow 0$,
 $\gamma_{wb}^{(s)}\rightarrow \frac{1}{2}$)                        & $3.5669$ & $4.7368$ & $5.7113$ & $6.8987$  \\ \hline\hline
\end{tabular}
\label{tab:liftsphnegtab1}
\end{table}

\begin{table}
\caption{Lowered upper bounds on $\alpha_{c,wb}$ -- higher $\alpha,f_{wb}$ regime; $\kappa=0$, optimized parameters}\vspace{.1in}
\hspace{-0in}\centering
\begin{tabular}{||c||c|c|c|c||}\hline\hline
  $f_{wb}$                                             & $0.13$   & $0.15$   & $0.18$   & $0.20$    \\ \hline\hline
 $\widehat{\xi_{wb}^{(l,lift)}}$                       & $7.155e-6$ & $3.165e-06$ & $3.785e-06$  & $7.876e-07$  \\ \hline
 $c_{3}^{(s)}$                                         & $1.2827$ & $1.6059$ & $2.1064$  & $2.4613$   \\ \hline
 $\gamma_{wb}^{(s)}$                                   & $0.2733$ & $0.2398$ & $0.1996$  & $0.1775$   \\ \hline
 $\nu_{wb}$                                            & $0.5971$ & $0.4338$ & $0.2748$  & $0.2043$   \\ \hline
 $\alpha_{c,wb}^{(u,low)}$                             & $\mathbf{7.1892}$ & $\mathbf{8.4974}$ & $\mathbf{10.9700}$ & $\mathbf{13.0802}$  \\ \hline\hline
 $\nu_{wb}$ ($c_3^{(s)}\rightarrow 0$,
 $\gamma_{wb}^{(s)}\rightarrow \frac{1}{2}$)                        & $1.2687$ & $1.0741$ & $0.8379$  & $0.7083$ \\ \hline
 $\alpha_{c,wb}^{(u)}$ ($c_3^{(s)}\rightarrow 0$,
 $\gamma_{wb}^{(s)}\rightarrow \frac{1}{2}$)                        & $7.5920$ & $9.2296$ & $12.5300$ & $15.5332$  \\ \hline\hline
\end{tabular}
\label{tab:liftsphnegtab2}
\end{table}

\begin{table}
\caption{Lowered upper bounds on $\alpha_{c,wb}$ -- lower $\alpha,f_{wb}$ regime; $\kappa=0.5$, optimized parameters}\vspace{.1in}
\hspace{-0in}\centering
\begin{tabular}{||c||c|c|c|c||}\hline\hline
 $f_{wb}$                                              & $0.15$   & $0.20$   & $0.23$   & $0.25$    \\ \hline\hline
 $\widehat{\xi_{wb}^{(l,lift)}}$                       & $0.0000$ & $1.161e-06$ & $1.325e-06$  & $2.548e-06$  \\ \hline
 $c_{3}^{(s)}$                                         & $0.0000$ & $0.2694$ & $0.6225$  & $0.85038$   \\ \hline
 $\gamma_{wb}^{(s)}$                                   & $0.5000$ & $0.4372$ & $0.3680$  & $0.3307$   \\ \hline
 $\nu_{wb}$                                            & $2.3606$ & $1.5159$ & $1.0425$  & $0.8225$   \\ \hline
 $\alpha_{c,wb}^{(u,low)}$                             & $\mathbf{2.6452}$ & $\mathbf{3.6298}$ & $\mathbf{4.3850}$ &
 $\mathbf{4.9772}$  \\ \hline\hline
 $\nu_{wb}$ ($c_3^{(s)}\rightarrow 0$,
 $\gamma_{wb}^{(s)}\rightarrow \frac{1}{2}$)                        & $2.3606$ & $1.7999$ & $1.5347$ & $1.3794$ \\ \hline
 $\alpha_{c,wb}^{(u)}$ ($c_3^{(s)}\rightarrow 0$,
 $\gamma_{wb}^{(s)}\rightarrow \frac{1}{2}$)                        & $2.6452$ & $3.6393$ & $4.4472$ & $5.1086$  \\ \hline\hline
\end{tabular}
\label{tab:liftsphnegtab3}
\end{table}

\begin{table}
\caption{Lowered upper bounds on $\alpha_{c,wb}$ -- higher $\alpha,f_{wb}$ regime; $\kappa=0.5$, optimized parameters}\vspace{.1in}
\hspace{-0in}\centering
\begin{tabular}{||c||c|c|c|c||}\hline\hline
  $f_{wb}$                                             & $0.28$   & $0.30$   & $0.33$   & $0.35$    \\ \hline\hline
 $\widehat{\xi_{wb}^{(l,lift)}}$                       & $4.632e-06$ & $2.195e-06$ & $2.778e-06$  & $1.422e-06$  \\ \hline
 $c_{3}^{(s)}$                                         & $1.1921$ & $1.4254$ & $1.7919$  & $2.0525$   \\ \hline
 $\gamma_{wb}^{(s)}$                                   & $0.2841$ & $0.2576$ & $0.2234$  & $0.2033$   \\ \hline
 $\nu_{wb}$                                            & $0.5830$ & $0.4656$ & $0.3329$  & $0.2659$   \\ \hline
 $\alpha_{c,wb}^{(u,low)}$                             & $\mathbf{6.0383}$ & $\mathbf{6.8916}$ & $\mathbf{8.4625}$ & $\mathbf{9.7620}$  \\ \hline\hline
 $\nu_{wb}$ ($c_3^{(s)}\rightarrow 0$,
 $\gamma_{wb}^{(s)}\rightarrow \frac{1}{2}$)                        & $1.1725$ & $1.0494$ & $0.8834$  & $0.7838$ \\ \hline
 $\alpha_{c,wb}^{(u)}$ ($c_3^{(s)}\rightarrow 0$,
 $\gamma_{wb}^{(s)}\rightarrow \frac{1}{2}$)                        & $6.3476$ & $7.3889$ & $9.398$ & $11.1434$  \\ \hline\hline
\end{tabular}
\label{tab:liftsphnegtab4}
\end{table}

\section{Conclusion}
\label{sec:conc}

In this paper we looked at storage capacities of spherical perceptrons. Differently from the standard case when one expects perfect storage of all patterns we here consider the case when errors in storing some of the patterns may be allowed. To mathematically characterize possible errors we represent them as a fraction of the total number of stored patterns that are allowed to be memorized incorrectly. This is essentially a classical setup of storage spherical perceptron type of memories with the so-called limited errors.

Various aspects of these types of memories have been studied throughout the literature. Here we focused on the storage capacities and provided a powerful set of mechanisms that can be used to quantify these capacities in a statistical context. We first introduced a powerful mechanism that enabled us to show that the predictions obtained for these types of capacities through the replica symmetric approach of statistical mechanics are in fact rigorous upper bounds on the true capacity values. We then presented a further refinement of the mechanism that can be used to actually lower these bounds in certain range of parameters of interest. This eventually indicates that the original problem may have a substantial underlying combinatorial structure that the replica symmetry may not be able to capture.

\begin{table}
\caption{Lowered upper bounds on $\alpha_{c,wb}$ -- lower $\alpha,f_{wb}$ regime; $\kappa=1$, optimized parameters}\vspace{.1in}
\hspace{-0in}\centering
\begin{tabular}{||c||c|c|c|c||}\hline\hline
 $f_{wb}$                                              & $0.20$   & $0.25$   & $0.30$   & $0.35$    \\ \hline\hline
 $\widehat{\xi_{wb}^{(l,lift)}}$                       & $0.0000$ & $0.0000$ & $0.0000$  & $3.320e-06$  \\ \hline
 $c_{3}^{(s)}$                                         & $0.0000$ & $0.0000$ & $0.0000$  & $0.2127$   \\ \hline
 $\gamma_{wb}^{(s)}$                                   & $0.5000$ & $0.5000$ & $0.5000$  & $0.4496$   \\ \hline
 $\nu_{wb}$                                            & $3.3916$ & $2.8039$ & $2.3238$  & $1.6771$   \\ \hline
 $\alpha_{c,wb}^{(u,low)}$                             & $\mathbf{1.3715}$ & $\mathbf{1.7398}$ & $\mathbf{2.2374}$ & $\mathbf{2.9259}$  \\ \hline\hline
 $\nu_{wb}$ ($c_3^{(s)}\rightarrow 0$,
 $\gamma_{wb}^{(s)}\rightarrow \frac{1}{2}$)                        & $3.3916$ & $2.8039$ & $2.3238$ & $1.9191$ \\ \hline
 $\alpha_{c,wb}^{(u)}$ ($c_3^{(s)}\rightarrow 0$,
 $\gamma_{wb}^{(s)}\rightarrow \frac{1}{2}$)                        & $1.3715$ & $1.7398$ & $2.2374$ & $2.9313$  \\ \hline\hline
\end{tabular}
\label{tab:liftsphnegtab5}
\end{table}

\begin{table}
\caption{Lowered upper bounds on $\alpha_{c,wb}$ -- higher $\alpha,f_{wb}$ regime; $\kappa=1$, optimized parameters}\vspace{.1in}
\hspace{-0in}\centering
\begin{tabular}{||c||c|c|c|c||}\hline\hline
  $f_{wb}$                                             & $0.40$   & $0.43$   & $0.47$   & $0.50$    \\ \hline\hline
 $\widehat{\xi_{wb}^{(l,lift)}}$                       & $3.358e-06$ & $1.476e-06$ & $5.605e-06$  & $2.594e-06$  \\ \hline
 $c_{3}^{(s)}$                                         & $0.6596$ & $0.9322$ & $1.3155$  & $1.6281$   \\ \hline
 $\gamma_{wb}^{(s)}$                                   & $0.3616$ & $0.3186$ & $0.2696$  & $0.2377$   \\ \hline
 $\nu_{wb}$                                            & $1.0496$ & $0.7950$ & $0.5470$  & $0.4103$   \\ \hline
 $\alpha_{c,wb}^{(u,low)}$                             & $\mathbf{3.8664}$ & $\mathbf{4.6040}$ & $\mathbf{5.8853}$ & $\mathbf{7.1643}$  \\ \hline\hline
 $\nu_{wb}$ ($c_3^{(s)}\rightarrow 0$,
 $\gamma_{wb}^{(s)}\rightarrow \frac{1}{2}$)                        & $1.5709$ & $1.3839$ & $1.1562$  & $1.0000$ \\ \hline
 $\alpha_{c,wb}^{(u)}$ ($c_3^{(s)}\rightarrow 0$,
 $\gamma_{wb}^{(s)}\rightarrow \frac{1}{2}$)                        & $3.9355$ & $4.7662$ & $6.2858$ & $7.8879$  \\ \hline\hline
\end{tabular}
\label{tab:liftsphnegtab6}
\end{table}

Many other features of the spherical perceptrons are also of interest. They relate to their memory capacities as well as to how these memories are functioning. The results that we presented can be utilized to characterize all of these features and we will present results in these directions elsewhere. Also, the results we presented relate to a particular statistical version of the spherical perceptron. Such a version is within the frame of neural networks/statistical mechanics typically called uncorrelated. As was the case with the results we presented in \cite{StojnicGardGen13} when we studied the basics of the spherical perceptrons, the results we presented here can also be translated to cover the corresponding correlated case. While on the topic of randomness, we should emphasize that strictly speaking we instead of typical binary patterns assumed standard normal ones. This was to done to make the presentation as easy as possible. As mentioned earlier in the paper (and as discussed to a much greater detail in \cite{StojnicHopBnds10,StojnicMoreSophHopBnds10}), all results that we presented easily extend beyond the standard Gaussian setup we utilized. A way to show that would be to utilize a repetitive use of the central limit theorem. For example, a particularly simple and elegant approach in that direction would be the one of Lindeberg \cite{Lindeberg22}. Adapting our exposition to fit into the framework of the Lindeberg principle is relatively easy and in fact if one uses the elegant approach of \cite{Chatterjee06} pretty much a routine. However, as we mentioned when studying the Hopfield and Little models \cite{StojnicHopBnds10,StojnicMoreSophHopBnds10,StojnicAsymmLittBnds11}, since we did not create these techniques we chose not to do these routine generalizations.

We should also mention that in this paper we primarily focused on the behavior of the storage capacity when viewed from an analytical point of view. In other words, we focused on quantifying analytically what the capacity would be in a statistical scenario. Of course, a tone of interesting questions related to this same problem arise if one looks at it from an algorithmic point of view. For example, one may wonder how easy is to actually determine the strengths of the bonds that do achieve the storage capacity (or to be more in alignment with what we proved here, a lower bound of the storage capacity). These problems are not that easy even when the errors are not allowed. For example, if errors are not allowed, and if $\kappa\geq 0$ then computing the bonds strengths essentially boils down to solving the feasibility problem given in (\ref{eq:defprobucor1}). This problem of course can easily be cast as a convex optimization problem and solved in polynomial time. However, already as $\kappa$ transitions to $\kappa<0$ regime the feasibility problem given in (\ref{eq:defprobucor1}) may not be as easy. On the other hand, when the errors are allowed one faces the same type of concern when $\kappa<0$. Moreover, when the errors are allowed even the ``easy" case $\kappa\geq 0$ may not be so easy any more. Designing the algorithms that can handle all these cases seems as a somewhat challenging and interesting problem. As we mentioned above, in this paper we were mostly concerned with certain analytical properties of the spherical perceptrons and consequently did not present any considerations in the algorithmic direction. However, we do mention that one can design algorithms similar to those designed for problems considered in \cite{StojnicUpperSec13}. Since an algorithmic consideration of spherical perceptrons is an important topic on its own, we will present a more detailed discussion in this direction in a separate paper.

\begin{singlespace}
\bibliographystyle{plain}
\bibliography{GardWrBitsRefs}
\end{singlespace}

\end{document}